\documentclass[oneside]{amsart}
\usepackage{amsmath, amssymb}
\usepackage{amsthm}
\newtheorem{theorem}{Theorem}[section]
\newtheorem{proposition}[theorem]{Proposition}
\newtheorem{lemma}[theorem]{Lemma}

\usepackage{graphicx}

\def\S{\mathbb{S} } 

\def\T{\mathbb{T} }

\def\Q{\mathbb{Q} } 
\def\R{\mathbb{R} } 
\def\Z{\mathbb{Z} } 
\def\nbd{neighborhood } 
\def\nbds{neighborhoods } 
\def\R{\mathbb{R} }

\title[Genericity for non-wandering surface flows]
{Genericity for non-wandering surface flows}
\author{Tomoo Yokoyama}
\date{\today}

\address{Department of Mathematics, Hokkaido University, 
Kita 10, Nishi 8, Kita-Ku, Sapporo, Hokkaido, 060-0810, Japan \\
}
\email{yokoyama@math.sci.hokudai.ac.jp}

\thanks{The author is partially supported 
by the JST CREST Program at Department of Mathematics,  
Kyoto University of Education.}

\begin{document}

\maketitle

\begin{abstract}
Consider the set $\chi^0_{\mathrm{nw}}$ of 
non-wandering continuous flows on 
a closed surface. 
% ($r = 0$ or $1$). 
Then 
such a flow can be approximated by 
regular non-wandering flows 
without heteroclinic 
%($\partial$-)saddle 
connections 
nor locally dense orbits in $\chi^0_{\mathrm{nw}}$. 
Using this approximation, 
we show that  
a non-wandering continuous flow 
on a closed connected surface 
is  topologically stable 
if and only if 
the orbit space of it is homeomorphic to 
a closed interval. 
Moreover we state 
%the characterization of 
%topologically stable non-wandering flow on 
%the sphere $\S^2$ (resp. projective plane $\mathbb{P}^2$, 
%Klein bottle $\mathbb{K}^2$) 
%and 
the non-existence of 
topologically stable non-wandering flows 
on closed surfaces which are not neither $\S^2$, $\mathbb{P}^2$, nor $\mathbb{K}^2$.  
\end{abstract}

\section{Introduction and preliminaries}
%In \cite{D}, 
%it is shown that 
%there are no exceptional minimal sets of $C^2$-flows on tori. 
%In \cite{S}, this result is generalized to the compact surface cases. 
%On the other hand, 
%there is 
%the Denjoy flow ( first constructed by Poincar\'e \cite{P}) 
%which is a $C^1$-flow 
%with an exceptional minimal set. 
%In this paper, 
%we show the non-existence of exceptional minimal sets of 
%continuous non-wandering flows on compact surfaces. 
%
%%
In \cite{NZ}, 
they have given the characterization of 
the non-wandering flows on compact surfaces with finitely many singular points. 
In \cite{Mar}, 
the author has given a description near orbits of the non-wandering flow with 
the set of singular points which is totally disconnected. 
%
%In this paper, 
%we deal with non-wandering flows 
%with arbitrary singular points.
%
On the other hand, 
in \cite{MW}, 
they have characterized 
structurally stable divergence free vector fields on 
connected compact orientable surfaces. 
In \cite{Y}, 
the author has given the characterization of 
a non-wandering flows on compact surfaces with  arbitrary singular points. 
In this paper, 
we study genericity of 
non-wandering flows 
%with finitely many singular points 
with arbitrary singular points 
on closed surfaces.  
Moreover 
we characterize the topological stability for 
non-wandering continuous flows 
on closed connected surfaces.  
%A non-wandering continuous flow 
%on a closed connected surface 
%is  topologically stable 
%if and only if 
%the orbit space of it is homeomorphic to 
%a closed interval. 
%On the other hands, 
%we characterize topologically the non-wandering property for 
%a continuous flow 
%with arbitrary singular points 
%on a closed surface. 
%
%We assume that 
%all surfaces are compact. 
% and connected.  

By flows, 
we mean continuous $\R$-actions on compact surfaces.
%we mean continuous $\R$-actions defined by vector fields on surfaces. 
Let 
$v: \R \times M \to M$ be a flow. 
% on a surface $M$. 
Put $v_t( \cdot ) := v(t, \cdot )$ 
and $O_v( \cdot ) := v(\R, \cdot )$.  
Recall that 
a point $x$ of $M$ is singular 
%point or a singular point 
if 
$O_v(x)$ is a singleton, 
is regular if it is not singular,  
is periodic if 
there is $T>0$ such that 
$v(T, x) = x$ and 
$v(t, x) \neq x$ for any $t \in (0, T)$, and 
%a point $x$ of $M$ is 
is non-wandering if 
for each \nbd $U$ of $x$ and 
each positive number $N$, 
there is $t \in \R$ with $|t| > N$ such that 
$v_t(U) \cap U \neq \emptyset$   
%(resp. $v_t(x) \cap U$) 
% is not empty, 
and that 
$v$ is non-wandering if 
every point is non-wandering. 
An orbit is proper if 
it is embedded, 
%An orbit is 
locally dense if 
the closure of it has nonempty interior, 
and 
exceptional if 
it is neither proper nor locally dense.  
A point is proper (resp. locally dense, exceptional) if 
so is its orbit. 
Denote by 
$\mathop{\mathrm{Sing}}(v)$ 
(resp. $\mathop{\mathrm{Per}}(v)$, $\Omega{(v)}$) 
the set of singular 
(resp. periodic, non-wandering) 
points of $v$ 
and by 
$\mathrm{LD}$ (resp. $\mathrm{E}$, $\mathrm{P}$)
the union of locally dense orbits (resp. exceptional orbits, 
non-closed proper orbits).  
Notice that 
$\mathrm{LD}$ need not open 
because of Example 1 \cite{Y}. 
%(see Example \ref{ex1}). 
%
Denote by 
$\chi^r_{\mathrm{nw}} = \chi^r_{\mathrm{nw}}(M)$ 
the set of non-wandering $C^r$ flows on a surface $M$ 
with respect to the $C^{r}$-topology for $r \geq 0$. 
Recall that 
the omega (resp. alpha) limit set of a point $x$ is  
$\omega(x) 
:= \bigcap_{n\in \mathbb{R}}\overline{\{v_t(x) \mid t > n\}}$ 
(resp.  $\alpha(x) := \bigcap_{n\in \mathbb{R}}\overline{\{v_t(x) \mid t < n\}}$) 
and that 
a separatrix of a singular point $p$ is 
a regular orbit of a point 
whose alpha or omega limit set 
is the singleton of $p$. 
A $\partial$-$k$-saddle (resp. $k$-saddle) is 
an isolated singular point on (resp. outside of) $\partial M$ with exactly $(2k + 2)$-separatrices, 
counted with multiplicity, 
where $\partial M$ is the boundary of $M$. 
A multi-($\partial$-)saddle is 
a ($\partial$-)$k$-saddle for some $k \in \dfrac{1}{2} \Z_{\geq 0}$. 
%For $k = 1$, it 
A $1$-saddle is topologically an ordinary saddle. 
%For odd $2k + 2$, it is locally non-orientable. 
The (multi-)saddle connection diagram 
is the union of (multi-)saddles, 
(multi-)$\partial$-saddles, and 
separatrices connecting (multi-)($\partial$-)saddles. 
A (multi-)saddle connection is 
a connected component of 
the (multi-)saddle connection diagram. 
A subset is essential 
if 
some connected component of it 
is neither null homotopic 
nor homotopic to a subset of the boundary $\partial M$. 
%, 
%where $\partial M$ is the boundary of $M$. 
%
%
%
Recall that 
a $C^r$ flow ($r>0$) is regular if it has no degenerate singular points 
and 
that 
a continuous flow is regular if each singular point has a \nbd 
which is topologically equivalent to a \nbd of a non-degenerate singular point. 
%By a sub-trajectory of a point $p$, 
%we mean a non-degenerate curve which 
%is contained in an orbit 
%and contains $p$ as an interior point. 
An orbit arc is 
a non-degenerate segment of an orbit of $v$. 
A point $p$ is called a transverse point to a curve $C$ 
if there are 
an orbit arc $\gamma$ containing $p$  
and a small open disk $U$ centered at $p$ 
such that 
$C \cap \gamma = \{ p \}$, 
$\gamma - \{ p \} \subset U \setminus C$, 
and 
$U \setminus C$ is two open disks 
each of which intersects $\gamma$. 
A point $p$ in a non-degenerate curve $C$ is called 
a tangency to $C$ for a flow $v$ if 
there are 
a small \nbd $N$ of $C$ 
and 
a small orbit arc $\gamma$ whose interior 
contains $p$ 
%and which is contained in an orbit of $v$ 
such that 
$C \cap \gamma = \{ p \}$
and 
one connected component of $N - C$ contains 
$\gamma - \{ p \}$. 
When $C$ is a subset of a boundary of a open subset $D$ of $M$ 
with $\gamma - \{ p \} \subset D$ 
(resp. $\gamma - \{ p \} \subset M - D$), 
a tangency $p$ is 
said to be interior (resp. exterior) with respect to $D$. 
An orbit arc of a flow $v$ is called 
a tangent arc to a curve $C$ if 
this arc is contained in $C$. 

\section{Structural stability}

An interval exchange transformation of a circle $\T^1 = \R/\Z$ 
is a left-continuous bijection, with finitely many discontinuity points,
that is piecewise a rotation.
For $\theta \in [0, 1)$, 
denote by $R_{\theta}$ the rotation by angle $\theta$ on $\T^1$. 
We use the following tool which is an analogy of Theorem 2.1 \cite{AB}. 

\begin{lemma}\label{lem02a}
Let $f$ be an interval exchange transformation. 
Then 
there is a dense subset $S$ of $[0, 1)$ 
such that 
the composition 
$f \circ R_{\theta}$ has periodic orbits 
for any $\theta \in S$. 
%, 
%where $R_{\theta}$ is the rotation by angle $\theta$ on $\T^1$. 
\end{lemma}

\begin{proof}
Suppose that 
$f$ is a rotation $R_{\alpha}$. 
Since the composition $f \circ R_{\theta}$ 
for any $ \theta \in [0, 1)$ with $\alpha + \theta \in \Q$ 
is a rational rotation, 
a dense subset $\{ \theta \in [0, 1) \mid \alpha + \theta \in \Q \}$ 
is desired. 
Thus we may assume that 
$f$ has discontinuity points.  
Replacing 
Theorem 3.1 \cite{AB} 
with 
Theorem 2 \cite{Mas} 
in the proof of Theorem 2.1 \cite{AB}, 
we can obtain this assertion. 
\end{proof}

Recall that 
a subset is 
($v$-)saturated 
or a ($v$-)saturation 
if 
it is a union of orbits of $v$. 
Moreover, 
a periodic orbit is 
one-sided (resp. two-sided) if 
there is (resp. is not) a small \nbd of it which  
is homeomorphic to a M\"obius band. 
We state the non-existence of 
locally dense orbits.

%
%From now on, 
%we assume that 
%$M$ is closed. 

\begin{lemma}\label{lem02}
Each flow $v$  
on a compact 
%orientable 
surface $M$ 
can be 
approximated by 
a flow without locally dense orbits 
in $\chi^r_{\mathrm{nw}}$ ($r \in \Z_{\geq 0}$)
such that it preserves 
$\mathop{\mathrm{Sing}}(v) \sqcup \mathop{\mathrm{Per}}(v)$. 
%, 
%where $\sqcup$ is the disjoint union symbol.   
% ($r \geq 0$).  
Moreover 
if $\mathrm{Sing}(v)$ is finite, 
then 
this perturbation can be chosen 
such that it preserves 
$\mathop{\mathrm{Sing}}(v) \sqcup \mathop{\mathrm{Per}}(v) \sqcup \mathrm{P}$.  
\end{lemma}

\begin{proof} 
We may assume that 
$M$ is connected. 
Fix any $v \in \chi^r_{\mathrm{nw}}$. 
By Lemma 2.4 \cite{Y}, 
there are no exceptional orbits 
and 
$\mathop{\mathrm{Per}}(v)$ is open. 
%Then 
%$\overline{\mathop{\mathrm{Per}}(v)} \subseteq 
%\mathop{\mathrm{Per}}(v) \sqcup \mathop{\mathrm{Sing}}(v) \sqcup \mathrm{P}$ 
%and so  
By Lemma 2.3 \cite{Y}, 
we have 
$\overline{\mathop{\mathrm{Per}}(v) \sqcup \mathop{\mathrm{Sing}}(v)} \cap \mathrm{LD} = \emptyset$. 
%
%By taking 
%%a double covering of $M$ and 
%the doubling of $M$, 
%we may assume that 
%%$v$ is transversally orientable 
%%and 
%$M$ is closed. 
By induction on the genus $g = g(M)$ of $M$, 
we will show the assertion. 
Suppose that $g = 0$. 
Trivially there are no locally dense orbits. 
Assume that 
$g >0$.   
Suppose that  
$M$ has an essential periodic orbit $\gamma$. 
%Consider an essential periodic orbit $\gamma$, 
Cutting this periodic orbit $\gamma$, 
then 
$M - \gamma$ has 
one or two new boundaries. 
%
%
%Suppose that 
%$M - \gamma$ has 
%one new boundary. 
%Then $M$ is non-orientable and 
%$\gamma$ has a small saturated \nbd $U$
%which consists of periodic orbits 
%and 
%is homeomorphic to the M\"obius band. 
%For an periodic orbit $\gamma' \neq \gamma \subset U$, 
%there is an annular \nbd of $\gamma'$ 
%which consists of periodic orbits.  
%%and the complement $M - \gamma'$ has two new boundaries. 
%Replacing $\gamma$ with $\gamma'$, 
%we may assume that 
%there is an annular \nbd of $\gamma'$ 
%which consists of periodic orbits. 
%%the complement $M - \gamma$ has two new boundaries. 
%
Adding one or two center disks to the two boundaries of $M - \gamma$,  
we obtain 
a new compact surface $M'$ each of whose connected components 
has the genus less than one of $M$ 
and the resulting non-wandering flow $v'$ on $M'$. 
Then the inductive hypothesis implies 
that 
$v'$ can be approximated by 
a flow $v'_1$ without locally dense orbits 
%in $\chi^r_{\mathrm{nw}}$ 
such that it preserves 
$\mathop{\mathrm{Sing}}(v') \sqcup \mathop{\mathrm{Per}}(v')$.   
Removing the new one or two center disks 
and pasting the boundaries of $M - \gamma$, 
the resulting flow from $v'_1$ is a desired perturbation of $v$.  
Thus we may assume that 
each periodic orbit of $v$ is not essential.  
For any locally dense orbit $O$, 
%we can take a small transversal arc $\gamma$ to a point of $O$ 
%such that $\gamma \cap (\mathop{\mathrm{Sing}}(v) \sqcup \mathop{\mathrm{Per}}(v)) = \emptyset$. 
%By 
the structure theorem \cite{G} implies that 
there is a closed transversal $C \subset \overline{O}$. 
Since $\mathop{\mathrm{Per}}(v)$ is open, 
the closure $\overline{O}$ contains no periodic orbits 
and so $C \cap \mathop{\mathrm{Per}}(v) = \emptyset$. 
Since $C$ is transversal, 
we obtain 
$C \cap (\mathop{\mathrm{Sing}}(v) \sqcup \mathop{\mathrm{Per}}(v)) = \emptyset$.
Let $T$ be the first return map on $C$ induced by $v$,  
$E$ the interval exchange transformation, 
and $h$ the semi-conjugate (i.e. $T$ covers $E$ via $h$) 
as in the structure theorem \cite{G}.  
Since the image of $T$ for each point of $C \cap O$ is dense in $C$, 
the domain of $T$ is open dense in $C$.  
Then the semi-conjugate $h$ is injective. 
Indeed otherwise there are two points $x \neq y \in C$ 
such that $h(x) = h(y)$. 
Since the degree of $h$ is one, 
the monotonicity of $h$ implies that 
there is an open interval $I$
% which contains $x, y$ and  
whose image of $h$ is a singleton.  
Since $C \cap O$ is dense in $C$, 
the image $h(C \cap O)$ is 
%a closed orbit and so 
finite. 
The Hausdorff separation axiom implies that  
the image of $h$ is finite and so 
a singleton, 
which contradicts to that the degree of $h$ is one.   
Therefore 
this $h$ becomes bijective and so homeomorphic.  
Thus 
the first return map $T$ on $C$ is 
topologically conjugate to 
an interval exchange transformation. 
%
%Recall the following fact \cite{BS} that 
%for any countable dense subsets $A$ and $B$ of $\R$, 
%there is an entire function $f: \C \to \C$ such that 
%$f(A) = B$. 
%Suppose that 
%$f_v$ is a rotation. 
%Then there is an arbitrary small number $\theta \in [0,1)$ 
%such that 
%the composition $f_v \circ R_{\theta}$ 
%is a rational rotation. 
%Thus we may assume that 
%$f$ has discontinuous points.  
%
%
By Lemma \ref{lem02a}, 
there is an arbitrary small number $\theta \in [0,1)$ 
such that 
the composition $T \circ R_{\theta}$ has 
periodic orbits. 
Since interval exchange transformations 
are area-preserving, 
this composition $T \circ R_{\theta}$ is area-preserving and so non-wandering. 
Taking an arbitrary small perturbation by the rotation $R_{\theta}$ near $C$, 
we obtain a non-wandering resulting flow $v'$ which has an essential periodic orbit. 
As above, 
we can obtain the desired perturbation. 
%
%Therefore the resulting flow $v'$ is desired.  

Suppose that $\mathrm{Sing}(v)$ is finite. 
%By Proposition 3.7 \cite{Y}, 
%the set $ \mathop{\mathrm{LD}}$ 
%is open 
%and so 
%$\mathop{\mathrm{Sing}}(v) \sqcup \mathop{\mathrm{Per}}(v) \sqcup \mathrm{P} = M -\mathrm{LD}$ 
%is closed. 
%
By Theorem 3\cite{CGL}, 
each singular point is 
either a center or a multi-saddle. 
By Corollary 2.9 \cite{Y}, 
the closed subset $\mathop{\mathrm{Sing}}(v) \sqcup \mathrm{P}$ 
consists of finitely many orbits. 
Let $D \subseteq \mathop{\mathrm{Sing}}(v) \sqcup \mathrm{P}$ be the subset of the multi-saddle connection 
which is the finite union of multi-saddles, multi-$\partial$-saddles,  
and separatrices in $\mathrm{P}$. 
%Then 
By Proposition 2.6 \cite{Y}, 
we have 
$D \sqcup \{ \text{center} \} \supseteq \mathop{\mathrm{Sing}}(v) \sqcup \mathrm{P}$. 
Assume that 
$D$ is essential. 
Let $\gamma \subseteq D$ be an essential simple closed curve. 
Then $\gamma$ contains $k$ multi-saddles for some $k \in \Z_{>0}$.   
Cutting this simple closed curve $\gamma$, 
then 
$M - \gamma$ has 
one or two new boundaries. 
\begin{figure}\label{fig01a}
\begin{center}
\includegraphics[scale=0.25]{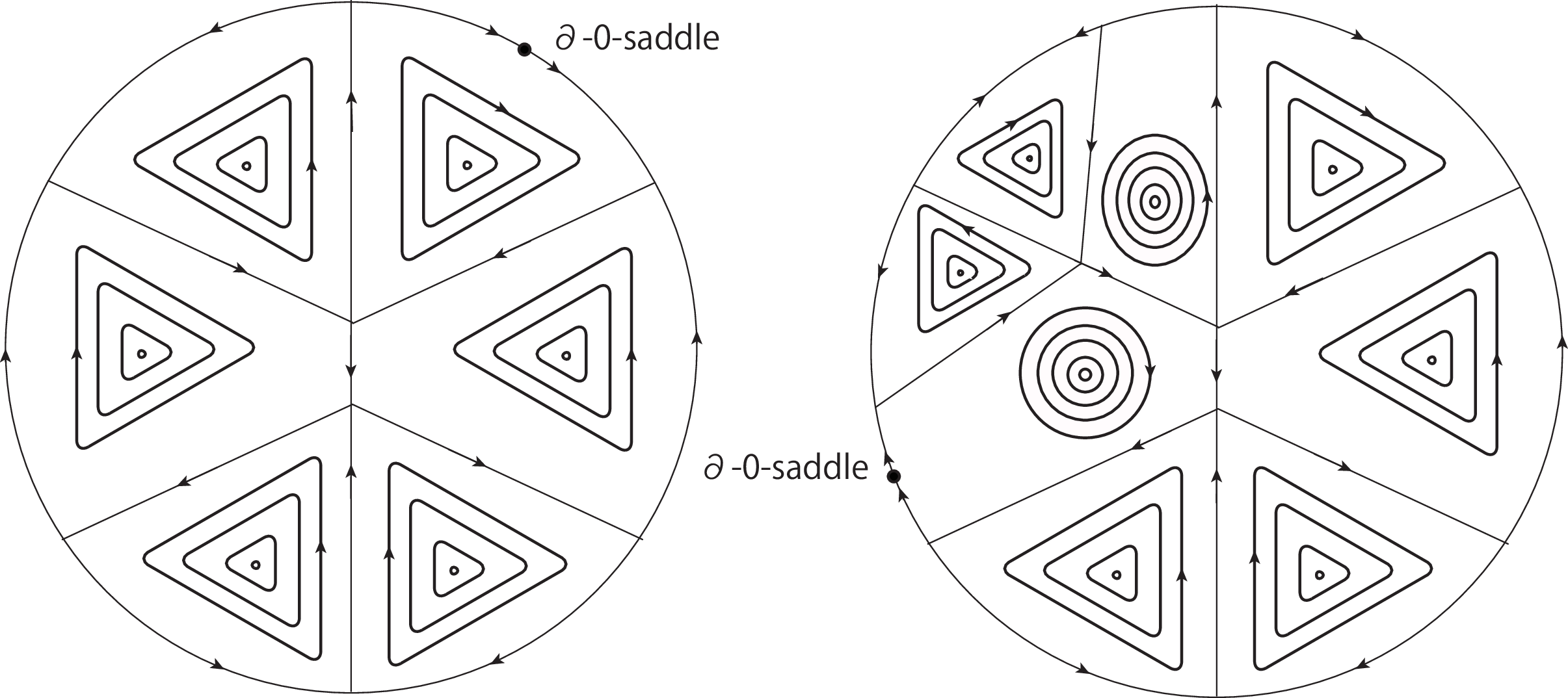}
\end{center} 
\caption{A disk with two (resp. three) $\partial$-saddles and one $\partial$-$0$-saddle}
%\caption{A disk with $2$ (resp. $3$) $\partial$-saddles and $1$ $\partial$-$0$-saddle}
\label{mdisk2}
\end{figure}
Adding one or two disks with $k - l$ $\partial$-saddles 
and $l$ $\partial$-$0$-saddles 
for some $l \in \Z_{\geq0}$
as in  Figure \ref{mdisk2} to the boundaries of $M - \gamma$,  
we obtain 
a new compact surface $M'$ which 
has the genus less than one of $M$ 
and the resulting non-wandering flow $v'$ on $M'$. 
Then the inductive hypothesis implies 
that 
$v'$ can be approximated by 
a flow $v'_1$ without locally dense orbits 
%in $\chi^r_{\mathrm{nw}}$ 
such that it preserves 
$\mathop{\mathrm{Sing}}(v') \sqcup \mathop{\mathrm{Per}}(v') \sqcup \mathop{\mathrm{P}}(v')$.   
Removing the new one or two disks
% with saddles 
and pasting the boundaries of $M - \gamma$, 
the resulting flow from $v'_1$ is a desired perturbation of $v$.  
Thus we may assume that 
$D$ is not essential.  
Then 
each simple closed curve in the multi-saddle connection diagram 
is non-essential. 
We can modify the closed transversal $C$ to satisfy 
$C \subset \mathrm{LD}(w)$.  
Indeed, 
let $\widetilde{D}$ be the union of $D$ and 
disks bounded by $D$. 
%By Corollary 2.9 \cite{Y}, 
Then 
this $\widetilde{D}$ 
consists of 
finitely many contractible closed disjoint subsets. 
Collapsing each connected component of $\widetilde{D}$ into 
a point, 
we obtain a new flow $w$ with $\mathrm{P}(w) = \emptyset$ 
and so we can modify the closed transversal $C$ such that 
$C \subset \mathrm{LD}(w)$. 
Blowing up collapsed points to original connected components, 
we have $C \subset \mathrm{LD}$.  
Thus the above perturbation can be chosen 
such that it preserves 
$\mathop{\mathrm{Sing}}(v) \sqcup \mathop{\mathrm{Per}}(v) \sqcup \mathrm{P} = M - \mathrm{LD}$.  
\end{proof}

The example in Example 1 \cite{Y} shows that 
the finiteness condition in Lemma \ref{lem02} 
is necessary. 
Recall that 
a homoclinic saddle connection is 
a saddle with two homoclinic separatrices,  
and 
a homoclinic $\partial$-saddle connection is 
a saddle connection which 
consists of $\partial$-saddles on a same boundary component 
and all separatrices of them which connect between 
the boundary component.  
A heteroclinic connection is 
a saddle connection 
with separatirices 
which connect 
either different boundaries 
or 
from a saddle and a different ($\partial$-)saddles. 
%a homoclinic saddle connection
%nor a homoclinic $\partial$-saddle connection.  
%a separatrix between different multi-($\partial$-)saddles. 
%$\partial$-saddles  with two homoclinic separatrices. 
%
We will show that 
arbitrarily small perturbations can break 
heteroclinic connections. 
The idea of the following proof is an analogy for the Hamiltonian vector fields  \cite{MW}.

\begin{lemma}\label{lem03}
Each continuous regular non-wandering flow can be approximated by 
%a sequence of 
a continuous regular flow whose complement of 
the set of periodic points 
is the union of finitely many centers and 
finitely many ($\partial$-)homoclinic saddle connections.  
%without heteroclinic connections.  
%in $\chi^r_{\mathrm{nw}}$ 
%forms a dense subset. 
% ($r \geq 0$). 
\end{lemma}

\begin{proof}
We may assume that 
$M$ is connected. 
Fix any regular flow $v \in \chi^0_{\mathrm{nw}}$. 
By Lemma \ref{lem02}, 
an arbitrary small perturbation makes 
$v$ into a flow without locally dense orbits. 
Thus we may assume that 
$v$ has no locally dense orbits. 
Then $\overline{\mathop{\mathrm{Per}}(v)} = M$
and 
$\partial \mathop{\mathrm{Per}}(v) =  \mathop{\mathrm{Sing}}(v) \sqcup \mathrm{P}$, 
where $\partial A := \overline{A} - \mathrm{int}A$ 
is the topological boundary of a subset $A \subseteq M$.  
By Corollary 2.7 \cite{G}, 
this non-wandering flow is topologically equivalent to 
a smooth flow. 
Using topologically equivalence, 
we may assume that 
$v$ is $C^{\infty}$. 
If $\mathop{\mathrm{Sing}}(v) = \emptyset$, 
then Corollary 2.9 \cite{Y} implies 
$\mathop{\mathrm{Per}}(v) = M$ 
and so $v$ is desired. 
Thus we may assume that 
$\mathop{\mathrm{Sing}}(v) \neq \emptyset$. 
First suppose that 
$M$ is closed. 
%
%Therefore 
%it suffices to show the $C^r$ case ($r \geq 1$) 
%and so we assume 
%$v$ is $C^1$. 
%%By Lemma \ref{lem00}, 
%%we may assume that 
%%$v$ is a regular non-wandering flow. 
%%By Theorem 3 \cite{CGL}, 
%%we have that 
Then 
each singular point is 
either a center or a saddle. 
%where 
%$\partial A := \overline{A} - \mathrm{int} A$ is the topological boundary of 
%a subset $A$ of $M$. 
%
By Corollary 2.9 \cite{Y}, 
each connected component of $\mathop{\mathrm{Per}}(v)$ is 
either an annulus or a M\"obius band.  
%a torus, 
%or a Klein bottle. 
%
Now we break heteroclinic connections by an arbitrary small 
perturbation. 
Indeed, 
fix a saddle point $p$ with a heteroclinic connection.  
Let $D$ be the saddle connection containing $p$. 
Then 
there is a saturated \nbd $V$ of $p$ such that 
each connected component of $V \setminus D$ 
is a saturated annulus or M\"obius band contained in $\mathop{\mathrm{Per}}(v)$. 
%and an arbitrary small \nbd $U$ of $p$. 
%By Hartman-Grobman Theorem, 
%Since the saddle connection diagram is 
%contained in $\partial \mathop{\mathrm{Per}}(v)$,  
Therefore 
there is an arbitrary small \nbd $U$ of $p$ 
%we may identify $U$ 
which can be identified with a square $[0,1]^2$ 
such that 
%.  
%%centered at the origin 
%%and radius $2$ in the $xy$-plane. 
%Also we may assume that  
%%$v$ is origin symmetric, 
$p$ is the origin $(0, 0)$ on $U$,  
the $x$-axis is the local stable manifold $W^s_v(p)$,  
and 
the $y$-axis is the local unstable manifold $W^u_v(p)$, 
and  that  
the set of orbits of $v$  is origin symmetric 
and 
axial symmetric with respect to 
the $x$-axis, 
the $y$-axis, 
and $I_{\pm}$
on $U$, 
where $I_{\pm} := \{ (x, \pm x) \mid x \in \R \}$.  
%Let $O_v(x, y)$ be the orbit through $(x, y)$ for $v$. 
%
Let $
I_{\sigma +} := \{\sigma1\} \times [0,1], 
%I_{-+} := \{ -1\} \times [0,1], 
I_{\sigma-} := \{\sigma1\} \times [-1, 0]$, 
%I_{--} := \{ -1\}  \times [-1, 0]$, 
$ 
J_{+\mu} := [0,1] \times \{ \mu1\}$, 
%J_{+-} := [0,1] \times \{ -1\}$, 
%J_{-+} := [-1, 0] \times \{ 1\}, 
and 
$J_{-\mu} := [-1, 0] \times \{ \mu1\}$ 
for any $\sigma, \mu  \in \{ -, +\}$.  
Then 
$\partial U = \bigcup_{\sigma, \mu  \in \{ -, +\}}(I_{\sigma \mu} \cup J_{\sigma \mu})$. 
Taking $U$ small, 
since orbits near $p$ move along the saddle connection containing $p$, 
we may assume that 
$v|_{M - \mathrm{int}U}$ maps 
$\mathrm{int} I_{\sigma \mu}$ 
for each $\sigma, \mu  \in \{ -, +\}$
into 
$\mathrm{int} J_{\sigma' \mu'}$ 
for some $\sigma', \mu'  \in \{ -, +\}$. 
Taking a reparametrization,  
we may assume that 
each orbit of $(\sigma c , \mu 1) \in \mathrm{int} I_{\sigma \mu}$ 
by $v|_{M - \mathrm{int}U}$ for each $c \in (0,1)$ 
goes back to 
$(\sigma' 1, \mu' c) \in J_{\sigma' \mu'}$. 
Then 
%We also may assume that 
%there is a small $a_0 > 0$ such that 
%$(1, b) \in O_v(b, 1)$ for $b \in (0, a_0)$ 
%%$(b, 1) \in O_v(1, -b)$ for $b \in (0, a_0)$,  
%and 
%%and that 
$U \cap (W^u_v(p) \cup W^s_v(p))  
= \{(0, y) \mid y \in (-1, 1) \} \cup 
\{(x, 0) \mid x \in (-1, 1) \}$. 
%Here $W^s_v(p)$ (resp. $W^u_v(p)$) is the stable (resp. unstable ) manifold of $p$ for $v$. 
%
%
Take a small open annulus $A$ centered at the origin. 
In fact, 
$A:= \{ (x, y) \mid x^2 + y^2 \in (\varepsilon/ 2, \varepsilon) \}$ 
for some small $\varepsilon > 0$. 
Let $v_{rot}$ be a $C^{\infty}$ flow whose support is $A$ 
and  each of whose regular orbits moves at a constant velocity 
and 
is a circle centered at the origin 
(i.e.  $\{ (x, y) \mid x^2 + y^2  =s \}$ for some $s  \in (\varepsilon/ 2, \varepsilon) $) 
such that 
the regular orbits of $v_{rot}$ move clockwise and 
$v' := v + v_{rot}$ is origin symmetric (see Figure \ref{rot2}), 
where $v + v_{rot}$ is the flow generated by a vector field $X_v + X_{v'}$. 
Here $X_w$ is the $C^{\infty}$ vector field whose flow is a $C^{\infty}$ flow $w$.  
Put $v'_U := v'|U$. 
%Fix a small number $a > 0$ such that 
Then there is a small number $a \in (0, 1)$ such that 
$\pm (1, a)
%, (-1, -a) 
\in W^s_{v'_U}(p)$, 
%and  
$\pm (a, 1)
%, (-a, -1) 
\in W^u_{v'_U}(p)$,  
%and 
%and that 
$(b, -1) \in O_{v'_U}(1, -b)$, 
and 
$(-b, 1) \in O_{v'_U}(-1, b)$  
for any $b \in(-a, a)$.  
\begin{figure}
\begin{center}
\includegraphics[scale=0.45]{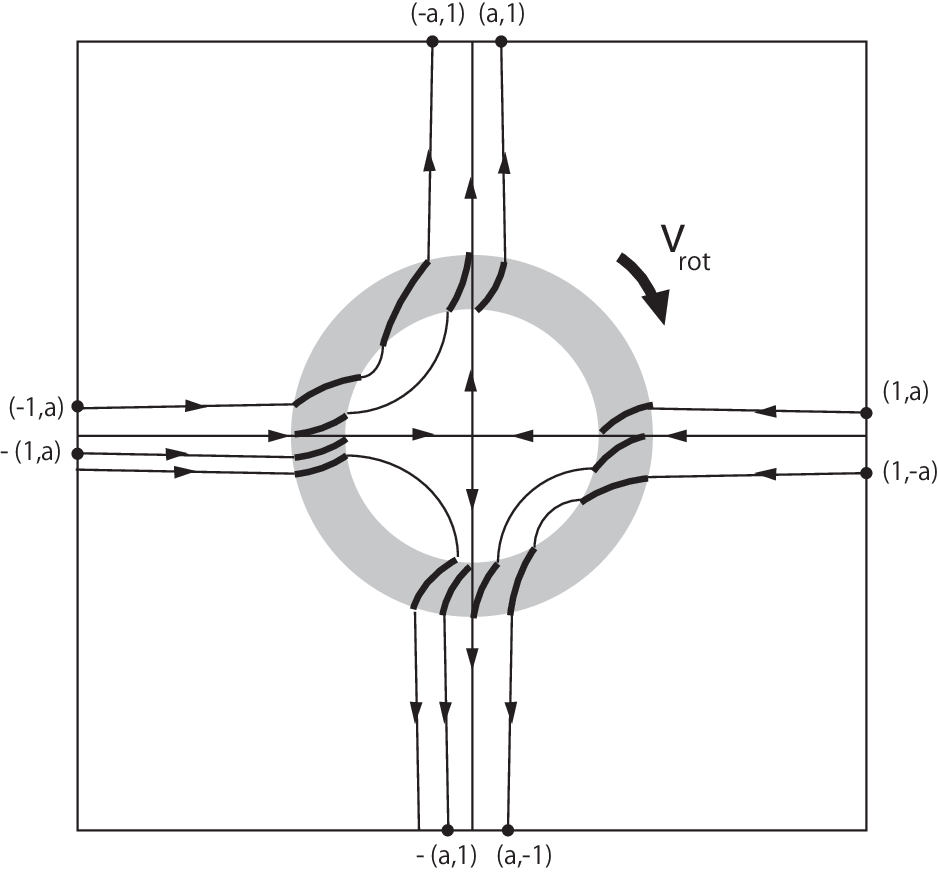}
\end{center}
\caption{Perturbation $v + v_{\mathrm{rot}}$}
\label{rot2}
\end{figure}
By taking $U$ small, 
we may assume that 
$U \cap (\cup_{\sigma, \sigma' \in \{ +, -\} } O_v( \sigma 1, \sigma' a ))$ 
consists of just four connected components. 
%\\
Since $O_v(\pm (1, \pm a))$ are periodic, 
we have that 
$W^s_{v'}(p)$ 
contains either $\pm (1, \pm a)$. 
Since $O_{v'}(\pm (1, \pm a))$ 
goes to either $p$ or $\pm (1, \pm a)$, 
we have that 
$W^u_{v'}(p)$ 
must go back to $p$. 
This means that 
$p$ is a homoclinic saddle connection. 
We will show that 
$v'$ is non-wandering. 
By symmetry, 
it suffices to show that 
each point contains in $\{ \pm (b, 1) \mid b \in (-a, a) \}$ 
is also periodic. 
Fix $b \in (0, a)$ and $\sigma, \mu \in \{ -, + \}$. 
By construction, 
we have that 
$- (\mu b, \sigma 1) \in O_{v'}(\sigma 1, \mu b)$ 
and  that 
$- (\mu b, \sigma 1)$ returns to either $(\sigma' 1, \mu' b)$  
for some $\sigma', \mu' \in \{ -, + \}$.   
%and so to either $\pm (b, \mp 1)$. 
%The periodicity of $\pm (b, \mp 1)$ with respect to $v$ 
This implies that 
$(\sigma 1, \mu b)$ is periodic with respect to $v'$. 
This shows that 
$v'$ is also non-wandering. 
Applying the above perturbation for all saddles, 
the resulting flow has no heteroclinic connections and 
is non-wandering. 
\begin{figure}
\begin{center}
\includegraphics[scale=0.45]{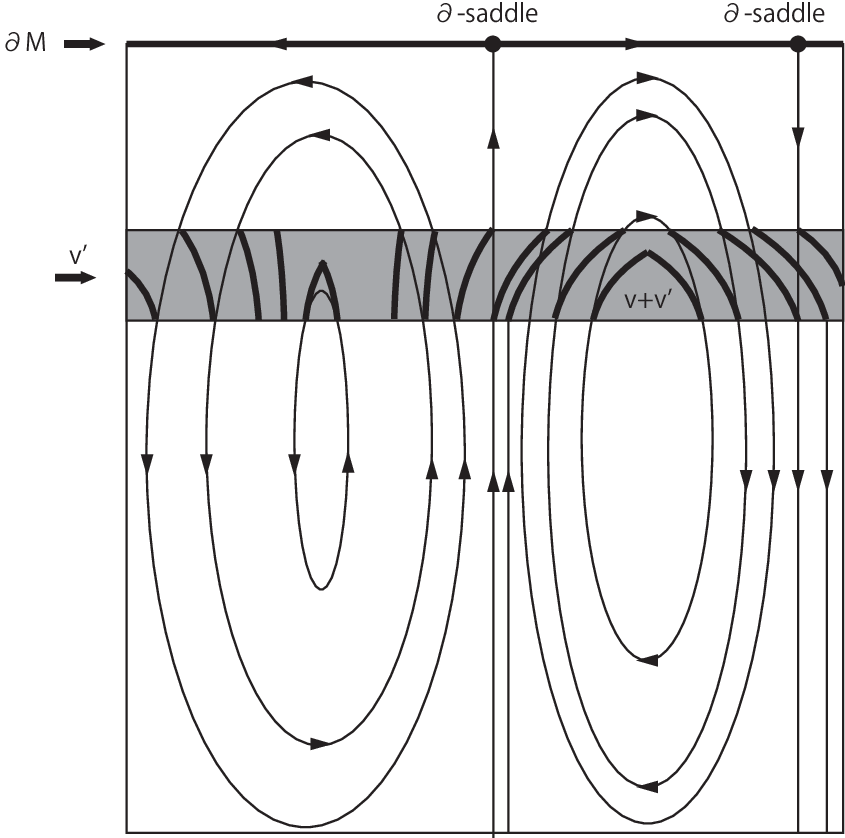}
\end{center}
\caption{Perturbation $v + v_{\mathrm{rot}}$ near the boundary}
\label{rot2a}
\end{figure}

Suppose that $M$ is not closed. 
As above, 
we may assume that 
there are no heteroclinic connections containing saddles. 
Thus it suffices to show that 
we can break a heteroclinic connection between $\partial$-saddles. 
Suppose that 
there is a heteroclinic connection $D$ between $\partial$-saddles. 
Then $D$ connects different boundaries. 
Considering a simple closed curve 
parallel to a boundary component of $D$ 
and taking a rotation as in Figure \ref{rot2a}, 
by a similar argument as above, 
we will show that 
we can break this heteroclinic connection $D$. 
%Iterating this process, 
%we can perturb $v$ into a desired flow. 
Indeed, 
let $B$ be a connected component of $\partial M$ 
contained in $D$. 
If there is a separatrix $\gamma$ 
connecting $\partial$-saddles on $B$, 
then collapse the disk bounded by $\gamma \cup B$.  
Thus we may assume that 
there are no separatrices between $\partial$-saddles  on $B$. 
Collapsing $B$ into one point, 
we obtain a new $k$-saddle $p$ for some $k \in \Z_{\geq 0}$. 
As above,  
there is a saturated \nbd $V$ of $p$ such that 
each connected component of $V \setminus D'$ 
is a saturated annulus or M\"obius band contained in $\mathop{\mathrm{Per}}(v)$, 
where $D'$ is the resulting saddle connection of $D$. 
By a coordinate change, 
there is an arbitrary small \nbd $U$ of $p$ 
which can be identified with an open unit disk 
$U = \{ (r \cos \theta, \sin \theta) \mid r \in [0,1), \theta \in [0, 2\pi) \}$ 
such that 
$p$ is the origin $(0, 0)$ on $U$,  
each $I_{4l}$ is the local stable manifold $W^s_v(p)$,  
and 
the $I_{4l +2}$ is the local unstable manifold $W^u_v(p)$, 
and  that  
the set of orbits of $v$  is origin symmetric 
and 
axial symmetric with respect to 
$I_{4l + m}$ 
on $U$ for any $l = 0, 1, \dots , k -1$ and $m =0, 1, 2, 3$,  
where $I_{4l + m} := \{ (r \cos \frac{(4l + m)2 \pi}{4k}, r \sin \frac{(4l + m)2 \pi}{4k}) \mid r \in [0,1) \}$.  
Let $\partial_j := \{ ( \cos (\theta + \frac{j \pi}{2k}), \sin  (\theta + \frac{j \pi}{2k})) \mid  \theta \in (0, \pi/2k) \} 
\subset \partial U$ and 
$T$ be the first return map on $\bigcup_{j =0}^{4k-1} \partial_j$ by 
$v|_{M - U}$.  
Replacing $U$ if necessary,  
we may assume that 
each $\partial_{4l +1}$ 
(resp. $\partial_{4l +2}$) is mapped by $T$ into some $\partial_{4l'}$ or $\partial_{4l' +3}$ isometrically.  
Let $v_{rot}$ be a $C^{\infty}$ flow whose support is a small open annulus in $U$ 
and  each of whose regular orbits moves at a constant velocity in 
the anti-clockwise direction  
and 
is a circle centered at the origin $p$.  
For any $(\cos \theta, \sin \theta)$ in the domain of $T$, 
write $(\cos \theta_T, \sin \theta_T) := T((\cos \theta, \sin \theta))$. 
Then 
$\theta \equiv \sigma \theta_T \mod \frac{\pi}{2k}$ 
for some $\sigma \in \{ -, + \}$.  
%for any $(\cos \theta, \sin \theta)$ in the domain of $T$. 
%Taking a slow rotation $v_{\mathrm{rot}}$ as above, 
Let $T'$ be the first return map on $\bigcup_{j =0}^{4k-1} \partial_j$ by 
$(v + v_{\mathrm{rot}})|_{U}$.  
Then 
there is a number $\theta_{a} \in (0, \pi/2k)$ 
such that 
$( \cos (\theta_a + \frac{4l \pi}{2k}), \sin  (\theta_a + \frac{4l  \pi}{2k}))$ 
is contained in the local stable manifold of $p$ for $v + v_{\mathrm{rot}}$ 
and 
$( \cos (- \theta_a + \frac{(4l +2) \pi}{2k}), \sin  (-\theta_a + \frac{(4l +2)  \pi}{2k}))$ 
is contained in the local unstable manifold of $p$.  
For any $(\cos \theta', \sin \theta')$ in the domain of $T'$, 
write $(\cos \theta'_{T'}, \sin \theta'_{T'}) := T'((\cos \theta', \sin \theta'))$. 
Then 
$\theta' \equiv  \sigma \theta'_{T'} \mod \frac{\pi}{2k}$ 
for some $\sigma \in \{ -, + \}$. 
As above, 
it's easy to check that 
the resulting flow $v + v_{\mathrm{rot}}$ 
is non-wandering and  
has less separatrices connecting $\partial$-saddles on different boundaries 
than $v$. 
%Iterating this process, 
%we obtain a desired flow. 
\end{proof}

%By  a $k$-saddle, 
%we mean a singular point   
%with $(2k + 2)$-separatrices. 
%
%Recall that 
%a subset is 
%($v$-)saturated 
%or a ($v$-)saturation 
%if 
%it is a union of orbits of $v$. 
%%
%Moreover, 
%a periodic orbit is 
%one-sided (resp. two-sided) if 
%there is (resp. is not) a \nbd of it which  
%is homeomorphic to a M\"obius band. 
We will show that 
regular non-wandering flows form 
a dense subset in $\chi^0_{\mathrm{nw}}$. 

\begin{lemma}\label{lem00} 
%$C^r$ regular non-wandering flows 
%form 
%a dense subset of $\chi^r_{\mathrm{nw}}$ 
%for each $r = 0$ or $1$. 
%%
%
Continuous non-wandering flows on a closed surface $M$ 
can be approximated by 
continuous regular non-wandering flows.  
\end{lemma}

\begin{proof} 
Fix any non-identical $v \in \chi^0_{\mathrm{nw}}$ 
and any distance $d$ induced by a Riemannian metric. 
As above, 
we may assume that 
there are no essential periodic orbits, 
because we can obtain a new non-wandering continuous flow 
by removing essential periodic orbits 
and pasting one or two center disks 
if exists.  
By Lemma \ref{lem02}, 
we may assume that 
$v$ has no locally dense orbits. 
Since $v$ is non-wandering, 
we have that 
%$\overline{\mathop{\mathrm{Per}}(v)} \supseteq M - \mathop{\mathrm{Sing}}(v)$ 
%and so that  
%%Then 
$M -(\mathop{\mathrm{Sing}}(v) \sqcup \mathop{\mathrm{Per}}(v)) = \mathrm{P}$ 
is nowhere dense. 
%
%By the flow box theorem  (cf. Theorem 1.1, p.45\cite{ABZ}), 
%taking a small perturbation for simple closed curves, 
%we may assume that 
%%such that 
%each tangency (resp. tangent arc) 
%is isolated in a simple closed curve 
%and so 
%each simple closed curve which contains no singular points 
%contains 
%at most finitely many tangencies and 
%tangent arcs. 
%
%
We will show that 
there is 
%are a number $\delta' > 0$ 
%and 
an open \nbd $U$ of $\mathop{\mathrm{Sing}}(v)$ 
with $\partial U \cap \mathop{\mathrm{Sing}}(v) = \emptyset$ 
such that 
$\partial U$ is homeomorphic to a finite disjoint union of circles 
and that 
each connected component of $\partial U$
either is a periodic orbit  
or consists of 
finitely many exterior tangencies contained in periodic orbits   
and of transverse points.  
%
%$\max_{t \in [-T, T]} d'_r( v_t|_{U} , \mathrm{id}) < \varepsilon/2$ 
%and 
%$\max_{t \in [-T, T]} d(v_t(x), y) < \varepsilon/2$  
%for 
%%any $t \in [-T, T]$ and for 
%any points $x, y \in U$ with $d(x, y) < \delta'$,  
Indeed, 
fix a small number $\varepsilon > 0$.  
%and a large number $T > 0$. 
%
%
%and a tangency to a curve $C$ means a point $p \in C$
%which has a small \nbd $N$ of it and a small number $s$ 
%such that  
%$C \cap \cup_{t \in (-s, s)}v_t(p) = \{ p \}$
%and 
%one connected component of $N - C$ contains $\cup_{t \in (-s, s)}v_t(p) \setminus C$. 
%Indeed, 
For $x \in  \mathop{\mathrm{Sing}}(v)$ and $r > 0$, 
let 
$B_{r}(x) := \{ z \in M \mid d(x, z) < r \}$. 
%Since $v|: [-t, t] \times M \to M$ for any $t \in \R$ is uniformly continuous, 
Then 
there is a large number $n > 0$
such that 
$\max \{  d(v_t(z), z) \mid t \in [-1/n, 1/n],   z \in B_{1/n}(x) \} < \varepsilon/2$.  
%Let $U_x := \{ y \in B_{1/n}(x) \mid \max_{t \in [-1/n, 1/n]} d( v_t(y) , y) < \varepsilon/2 \}$. 
% 
%and $D_{\e,t}(x) := \overline{B_{\e}(x)} \times [-t, t]$. 
%Since 
%$M \times [-T, T]$ and $\mathop{\mathrm{Sing}}(v)$ are compact, 
%%
%%$v|: [-T, T] \times M \to M$ is uniformly continuous, 
%there is a large number $N > 0$ such that 
%$\max \{  d_r(v_t|_{D_{1/N, T}(x)}, \mathop{\mathrm{id}}_{D_{1/N, T}(x)}) \mid (x, t) \in \mathop{\mathrm{Sing}}(v) \times [-T, T] \} < \varepsilon/2$, 
%where $d_r$ is the $C^r$ distance on the space of embedding from a cube ${D_{1/N, T}(x)}$ to $M \times [-T, T]$. 
%%, 
%%where 
%%$d'_0( v_t|_{U} , \mathrm{id}) := \max_{x \in U} d(v_t(x), x)$.  
%%and 
%%$d'_1( v_t|_{U} , \mathrm{id}) := \max_{x \in U} \{ d(v_t(x), x) ,
%%||D(v_t)_x|| - 1 \}$.  
%%Here $|| \cdot ||$ is the operation norm.  
%% 
%Let $U_x := B_{1/N}(x) \neq \emptyset$. 
%%\{ y \in B_{1/n}(x) \mid \max_{t \in [-T, T]} d( v_t(y) , y) < \varepsilon/2 \}$. 
%%By construction of $n$, 
%%we have  $U_x \neq \emptyset$. 
Let $U'_x \subset B_{1/n}(x)$ be an open \nbd of $x$ 
such that 
$\partial U'_x$ is homeomorphic to a circle 
and that 
$\overline{U'_x} \subset B_{1/n}(x)$. 
Since $\mathop{\mathrm{Sing}}(v)$ is compact, 
there are finitely many points $x_1, \dots , x_k \in \mathop{\mathrm{Sing}}(v)$ 
such that 
$U := U'_{x_1} \cup \cdots \cup U'_{x_k} \supset \mathop{\mathrm{Sing}}(v)$. 
Then 
$U$ is an open submanifold of $M$  
with 
$\partial U \cap \mathop{\mathrm{Sing}}(v) = \emptyset$ 
%$\partial U$ consists of finitely many disjoint circles 
such that 
$\partial U$ is homeomorphic to a finite disjoint union of circles. 
We may assume that $U$ 
has no orbit arcs in $U$ connecting different boundaries of $U$. 
Indeed 
assume that 
there is an orbit arc $\gamma$ 
connecting different boundaries of $U$.  
Since $\overline{\mathop{\mathrm{Per}}(v)} \supseteq M - \mathop{\mathrm{Sing}}(v)$, 
we may assume that 
$\gamma$ is 
two-sided and 
contained in a periodic orbit $O$. 
%Let $\gamma \subset O$ be 
%an arc which is contained in $U$ 
%and 
%connects different boundaries of $U$. 
%whose boundary is contained in $\partial U$. 
By the flow box theorem, 
% (cf. Theorem 1.1, p.45\cite{ABZ}), 
there is a thin saturated \nbd $V \subset \mathop{\mathrm{Per}}(v)$ of $O$. 
By removing the connected component $W$ of $U \cap V$ containing $\gamma$, 
the new \nbd $U - \overline{W}$ has less punctures than $U$. 
Thus finite removing operations imply that 
the resulting new \nbd has no periodic orbit connecting different boundaries. 
Replacing $U$ with this, 
the new \nbd $U$ has no orbit arc in $U$ connecting different boundaries of $U$. 

Let $O_j \subseteq \partial U$ be the periodic orbits.  
By removing a small color of $O_j$ consisting of  periodic orbits, 
we may assume that 
each $O_j$ are a connected component of $\partial U$. 
Then we will show that  
each connected component of 
$\partial U - \bigsqcup_j O_j$ has tangencies. 
Otherwise 
there is a connected component $C$ of $\partial U  - \bigsqcup_j O_j$ 
which has no tangencies. 
The non-wandering property 
implies that 
there is a point $x \in C$ which returns to $C$. 
This means the positive orbit of $x$ meets 
a different connected component of 
$\partial U$, which contradicts to the hypothesis. 

We may assume that 
there is no orbit arc $\gamma$ in $U$ which connects $\partial U$ 
such that $U - \gamma$ is connected,   
%(i.e. $[\gamma] \neq 0 \in \pi_1(U, \partial U)$), 
by removing the small open \nbd of $\gamma$. 
Moreover 
there is an open \nbd $V$ of $\partial U$  
which is homeomorphic to a finite disjoint union of open annuli  
such that $\overline{V} \cap \mathop{\mathrm{Sing}}(v) = \emptyset$. 
%Let $T$ be the set of tangencies on $\partial U$. 
%Since $ \mathop{\mathrm{Per}}(v)$ is open, 
%we obtain 
%$\partial U \setminus  \mathop{\mathrm{Per}}(v)$ is compact. 
Then there are 
connected open transverses $\gamma_1, \dots , \gamma_l$ in $V$ 
such that 
the saturation $W$ of $\gamma_1\cup \cdots \cup \gamma_l$ by $v|_{V}$ 
%has at most finitely many connected component 
%and 
contains $\partial U -  \bigsqcup_j O_j$. 
Then $\partial U \setminus W = \bigsqcup_j O_j$. 
%\subset \mathop{\mathrm{Per}}(v)$. 
Take a transverse arc with or without
%a transverse arc except 
one tangency 
in the saturation of each $\gamma_i$  by $v|_{V}$ 
as Figure \ref{transversal}. 
\begin{figure}\label{fig01b}
\begin{center}
\includegraphics[scale=0.4]{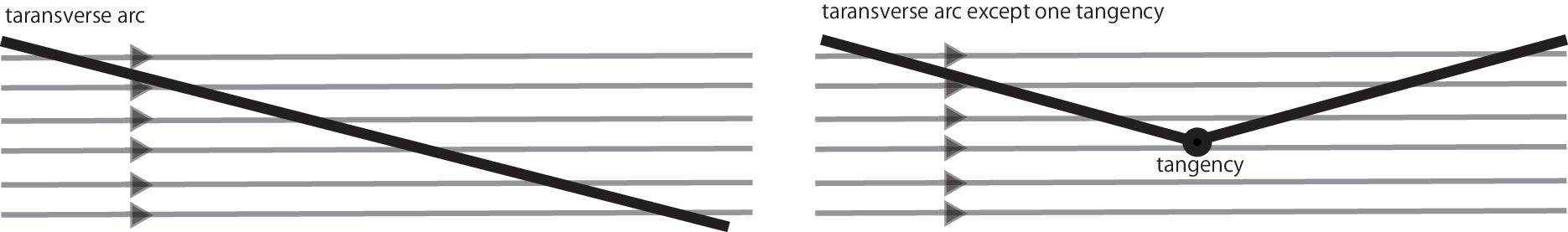}
\caption{a transverse arc and a transverse arc except one tangency }
\label{transversal}
\end{center}
\end{figure}
Connecting these transverse arcs (with one tangency), 
we can obtain a simple closed curve $C_i$ 
for each connected component of $\partial U - \bigsqcup_j O_j$ 
such that 
$C_i \subseteq W  \cap (V \cup U)$ 
has at most finitely many tangencies 
such that 
$\sqcup _i C_i$ is homotopic to 
%one of connected components of 
$\partial U - \bigsqcup_j O_j$ in $V \cup U$. 
Since $V$ is small, 
replacing $U$ with the open \nbd of $\mathop{\mathrm{Sing}}(v)$ 
whose boundary is 
$\bigsqcup_j O_j \sqcup \bigsqcup_i C_i$, 
%$U \cup {V}$, 
we may assume that $\partial U$ has at most finitely many tangencies. 
Since $\overline{\mathop{\mathrm{Per}}(v)} \supseteq M - \mathop{\mathrm{Sing}}(v)$, 
modifying $U$, 
we may assume that 
each tangency in $\partial U$ is 
two-sided 
and 
contained in periodic orbits. 
%
%By Corollary 2.9 \cite{Y}, we may assume that each tangency has a 
%saturated annular \nbd by a small perturbation. 
%
%
We will show that 
%an arbitrarily small perturbation 
we can modify $U$ to has no interior tangencies 
but has desired properties. 
Indeed, 
suppose that 
there is an interior tangency $p$ on 
a connected component $U'$ of $U$. 
First assume that 
$O_v(p) - \{ p \} \subset U'$. 
Fix a small closed saturated annular \nbd $A$ of $p$. 
Then 
$U' \setminus A$ has two boundaries which intersect to $A$ 
such that 
one of them is a periodic orbit 
and 
another boundary consists of 
one tangent arc, 
transverse points, 
and finitely many tangencies. 
Deduce $A$ to $B$ (see Figure \ref{B} and \ref{UB}), 
we can obtain a new open \nbd $U'' := U' \setminus B$ 
such that $U''$ has less interior tangencies than $U'$ 
and that 
one of the two new boundaries is a periodic orbit 
and 
another boundary consists of 
transverse points and finitely many tangencies. 
\begin{figure}\label{fig02}
\begin{center}
\includegraphics[scale=0.25]{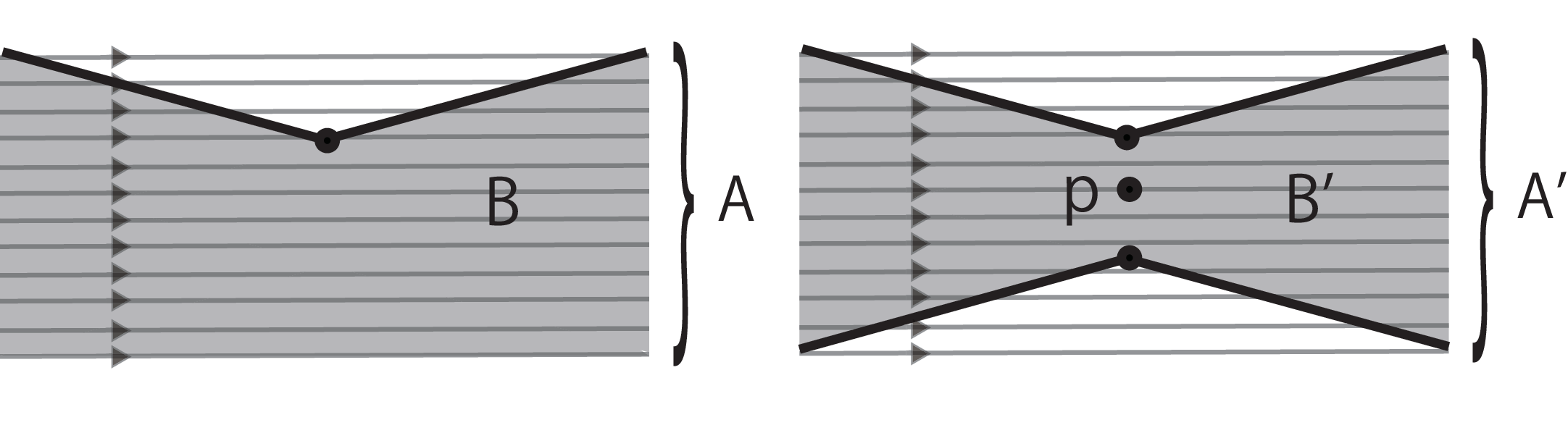}
\caption{A reduced \nbd $B$ (resp. $B'$) of $O_{v|_U}(p)$ }
%from $A$ (resp. $A'$)}
\label{B}
\includegraphics[scale=0.25]{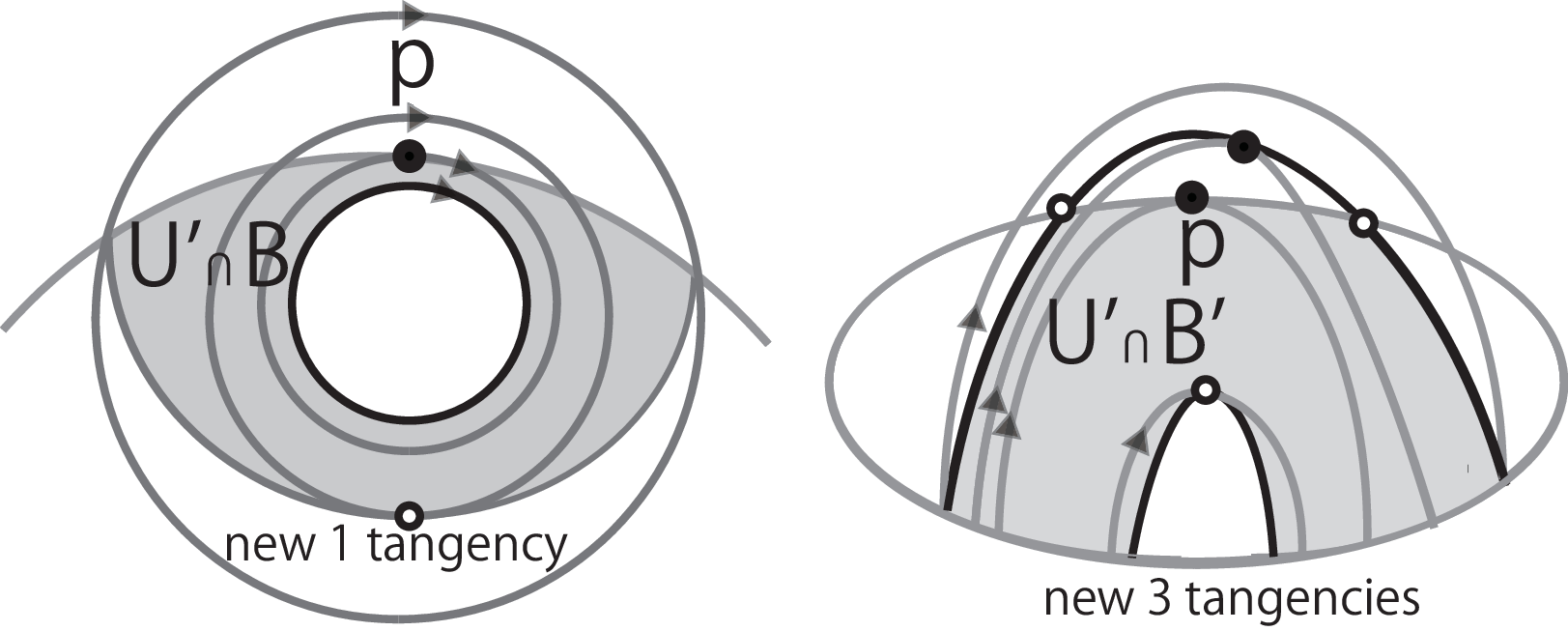}
\caption{A removing subset $U \cap B$ (resp. $U \cap B'$)}
\label{UB}
\end{center}
\end{figure}
%Note that 
%$U \cap B$ is an open annulus. 
%
Replace $U'$ with this \nbd $U''$. 
Second assume that 
$O_v(p) - \{ p \} \not\subseteq U'$. 
Then 
$O_v(p)$ intersects $\partial U'$.  
By the hypothesis, 
the orbit $O_v(p)$ connects the same boundary. 
Fix a small closed saturated annular \nbd $A'$ of $p$. 
%Then 
%$U' \setminus A$ consists of three components. 
Deducing $A'$ to $B'$ (see Figure \ref{B} and \ref{UB}), 
we can obtain a new open \nbd $U'' := U' \setminus B'$ 
such that $U''$ has less interior tangencies than $U'$ 
and that 
each of the two new boundaries consists of 
transverse points and finitely many tangencies.  
Replace $U'$ with this \nbd $U''$. 
By iterating this process, 
this $U'$ becomes desired.  
%because of $\overline{\mathop{\mathrm{Per}}(v)} \supseteq M - \mathop{\mathrm{Sing}}(v)$. 
Therefore 
applying this process to all boundaries of $U$, 
we obtain a new open \nbd  
each of whose boundary 
either is a periodic orbit 
or 
consists of 
transverse points and 
finitely many exterior tangencies. 
Hence this \nbd is a desired \nbd $U$.  

%
%
%Since $\overline{\mathop{\mathrm{Per}}(v)} \supseteq M - \mathop{\mathrm{Sing}}(v)$, 
%modifying $U$, 
%we may assume that 
%each tangency in $\partial U$ is contained in periodic orbits. 
%
Since $v|_{U}$ is small, 
if we replace $v|_{U}$ with a continuous flow on $U$ which is near to the identity 
and corresponds to $v$ near $\partial U$,  
then the resulting flow is near $v$. 
Therefore 
it suffices to show that 
there is a non-wandering regular continuous flow $w$ such that 
$w = v$ outside of $U$ and near $\partial U$.  
We may assume that 
there are no orbit arcs $\gamma$ in a connected component $U'$ of $U$ 
connecting $\partial U'$ 
such that each connected component of 
$\partial U' - \partial \gamma$ contains at least 
two tangencies. 
Indeed, otherwise 
there is such an arc $\gamma$ in $U'$. 
Since $\overline{\mathop{\mathrm{Per}}(v)} = M$, 
we may assume that 
$\gamma$ is contained in a periodic orbit. 
Taking a thin saturated \nbd $A'$ of $\gamma$ 
and reducing $A'$ into $B'$ as in Figure \ref{B}, 
the new \nbd $U' \setminus B'$ has less connecting arcs 
than $U'$ as above. 
Iterating this process, 
we can obtain a desired neighborhood.
Moreover 
we may assume that 
each orbit containing tangencies  for $U$ 
has no transverse points for $U$. 
Indeed, 
otherwise 
there is a periodic orbit $O$ containing a tangency $p$ 
on a connected component $U'$ of $U$  
has a transverse point on a connected component $U''$ of $U$. 
Taking a thin saturated \nbd $A$ of $O$, 
consider new \nbds $U' \setminus A$ and 
$U'' \setminus A$.  
Then the new \nbd $U' \setminus A$ 
has exactly one tangent arc. 
Reducing $A$ into $B$ near $p$, 
the boundary of $U' \setminus B$ consists of 
transverse points and finitely many exterior tangencies. 
On the other hand, 
$U'' \setminus A$ consists of two connected components $U''_1, U''_2$ 
such that 
$\partial U''_1$ has exactly one tangent arc  
and that 
$\partial U''_2$ consists of one tangency $p$, one tangent arc, and transverse points. 
Replacing $v|_{U''_2}$ 
with a non-singular flow $v''$ on $U''_2$ such that 
$v = v''$ near $\partial U'_2$ in $U''_2$, 
we may assume that 
$U'' \setminus A$ consists of one connected component 
$U''_1$ which has exactly one tangent arc. 
As above, reducing $A$ to $B$, 
we can obtain a new \nbd $U'' \setminus B$ consists of 
transverse points and finitely many exterior tangencies. 
Then the new reduced \nbd of $U$ 
has less orbits containing tangencies than $U$. 
Iterating this process, 
we can obtain a desired neighborhood.

Let $T$ be the union of periodic orbits containing tangencies for $U$. 
Then 
each connected component of $M - (T \cup U)$ is 
homeomorphic to either 
an annulus consisting of periodic orbits, 
a M\"obius band consisting of periodic orbits,  
or a flow box 
(i.e. a box $\{ (x, y) \mid x, y \in [0,1] \}$ with a vector field $\partial / \partial x$).  
In particular, 
each connected component of $M - (T \cup U)$ whose boundary contains transverse points of 
$\partial U$ is a flow box.  
%
%and 
%each saddle connection is closed with respect to $w$.  
%
%Since each exterior tangency of $U$ is contained in a periodic orbit, 
%there is a thin saturated \nbd $A$ of exterior tangencies of $U$ 
%such that  
%each connected component of $M - (A \cup U)$ is 
%homeomorphic to 
%either 
%an annulus consisting of periodic orbits, 
%a M\"obius band consisting of periodic orbits,  
%or a flow box 
%(i.e. a box $\{ (x, y) \mid x, y \in [0,1] \}$ with a vector field $\partial / \partial x$).  
%In particular, 
%each connected component of $M - (A \cup U)$ whose boundary contains transverse points of 
%$\partial U$ is a flow box.  
Thus we obtain a multi-graph $G$ whose 
vertices are connected components of $U$ 
and whose edges are connected components of $M - (T \cup U)$. 
Therefore 
if we obtain a new saddle connection diagram $D$ in $U$ 
such that 
each connected component of $\partial U \setminus T$ intersects 
exactly one separatrix in $U$, 
then we can modify $D$ in $U$ using $G$ 
%this saddle connection diagram in $U$ 
such that  
each separatrix in $D$ is  connected by an orbit of $v|_{M-U}$ 
to a separatrix in $D$ and so that 
the complement of the resulting saddle connection diagram in $M$ 
consists of annuli, M\"obius bands, and center disks. 
Hence we can 
%fill up in $U$ and so 
obtain a desired sufficiently slow continuous non-wandering flow $w$.  
%
%
%
%
%
%Note that 
%the orbit of each transverse point of $\partial U$ by $v|_{M - U}$ 
%connects between $\partial U$.  
%Notice that 
%if we obtain a new saddle connection diagram in $U$, 
%%
%then 
% 
%such that each separatrix in the saddle connection diagram in $M$ 
%connects between saddles, 
%then we can construct 
%a desired sufficiently slow continuous flow $w$, 
%because 
%$M - U \subseteq \overline{\mathop{\mathrm{Per}}(v)}$ and 
%$\partial U$ consists of transverse points and finitely many exterior tangencies. 
Therefore 
it suffices to show that 
there is such a saddle connection 
which is compatible with the 
orientation of $v|_{\partial U}$. 
Indeed, 
since $U$ has at most finitely many connected components, 
we may assume that 
$U$ is connected. 
By induction on the number $c_0$ of tangencies on $\partial U$, 
we will show this assertion.  
Note that 
the numbers of inward and outward transverse intervals are equal 
and so $c_0$ is even. 
Let $g$ be the genus 
%(resp. non-orientable genus) 
of $M$, 
$b$ the number of connected components of $\partial U$,  
$b_t$ the number of connected components of $\partial U - \bigsqcup_j O_j$,  
and $b_p$ the number of connected components of $\bigsqcup_j O_j$. 
Then $b = b_t + b_p$. 
%
%
%first we consider the following case:  
Suppose that 
$U$ is orientable.

1. case ( $g = 0$ and $c_0 = 0$): 
% $\partial U$ consists of periodic orbits) 
%First, consider the case $U$ has genus $0$.
%%
%%
%Suppose 
%each connected component of $\partial U$ is a periodic orbit. 
Suppose that 
$b = 0$. 
Then $U = M$ is a sphere 
and so 
we can replace $v|_{U}$ with 
a flow with two centers and periodic orbits on $M$.  
Suppose that 
$b = 1$. 
Then we replace $U$ with a center disk 
and obtain a desired flow on $U$. 
Suppose that 
$b > 1$. 
Let $b_-$ ($b_+$) be the 
number of periodic orbits in $\partial U$ 
 whose flow directions 
are clockwise (resp. anti-clockwise). 
Reversing the time if necessary, 
we may assume that 
$b_+ > 0$. 
Let $D$ be a center disk whose flow direction is 
anti-clockwise. 
Replace $b_+ - 1$ (resp. $b_-$) periodic orbits with 
$b_+ - 1$ figure eight 
(resp. $b_-$ inverse figure eight) 
saddle connections with a center disk 
and 
replace new centers with boundaries 
as Figure \ref{disk02}.
\begin{figure}\label{disk01}
\begin{center}
\includegraphics[scale=0.25]{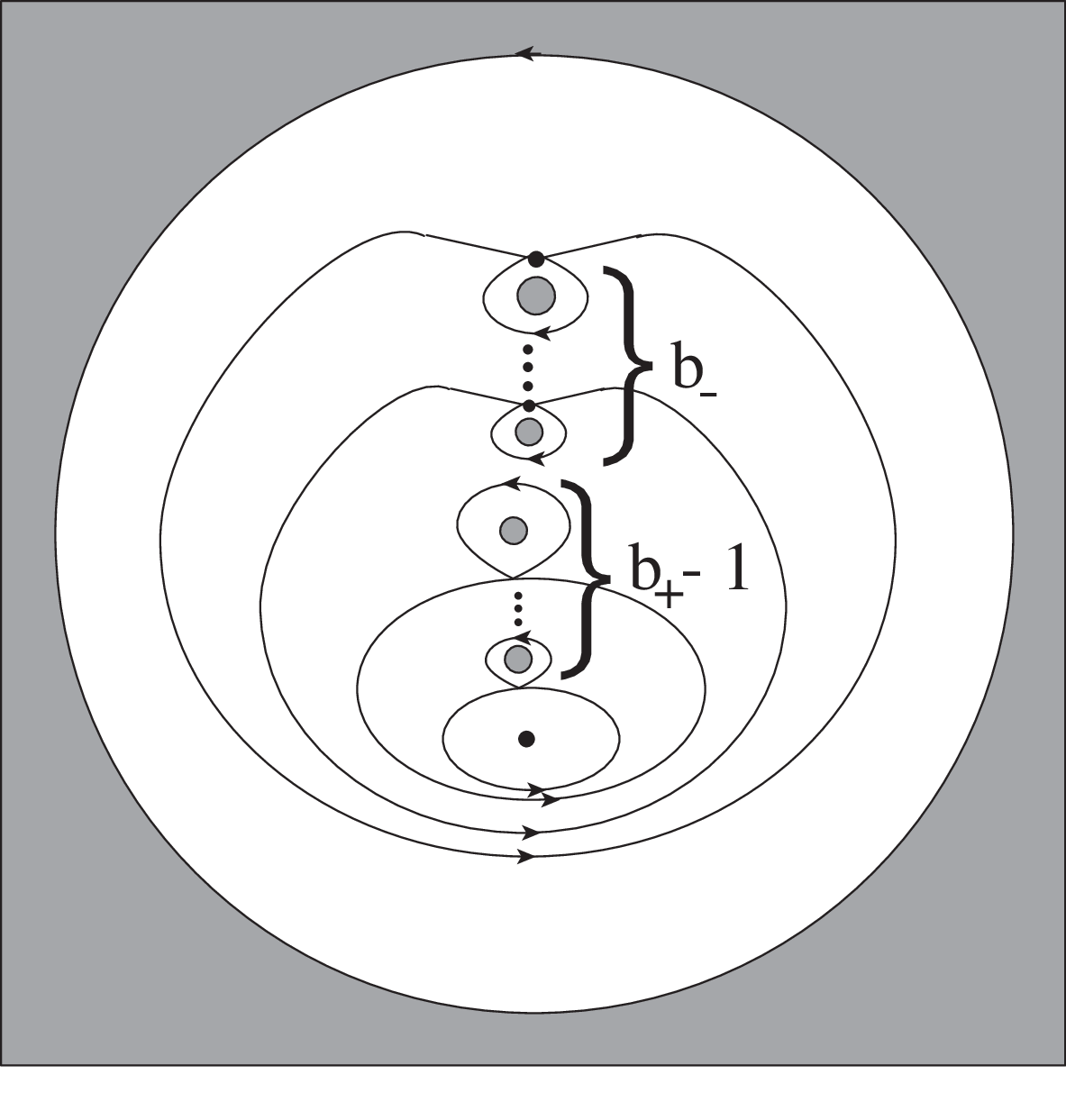}
\end{center}
\caption{A new $b -1$ punctured disk $D'$}
\label{disk02}
\end{figure}
Then we obtain a new $b -1$ punctured disk $D'$
whose boundaries are periodic orbits. 
Hence we replace $U$ with $D'$ 
and obtain a desired flow on a $b-1$ punctured open disk. 

2. case ( $g = 0$ and $c_0 > 0$ and $\bigsqcup_j O_j = \emptyset$): 
%each boundary of $\partial U$ has tangencies): 
%Thus 
%we may assume that 
%there is a connected component of $\partial U$ which 
%consists of transverse points and finitely many exterior tangencies. 
%
%
%Assume that 
%there are no boundaries of $\partial U$ which are periodic orbits. 
Suppose that $b = 1$. 
%$U$ has just one boundary. 
Then 
we can fill up $U$ by 
the flow which has just one singular point, 
which is a $(c_0 -2)/2$-saddle (see Figure \ref{rot}). 
By a small perturbation, 
we obtain a flow $v'$ with $(c_0 -2)/2$ (usual) saddles, 
which is desired. 
%Note that 
%we can choose that 
%each separatrix in the saddle connection diagram connects between saddles 
%by modifying connecting orbits in $U$, 
%and so that the resulting flow is non-wandering. 
%
%
%
Suppose that $b = 2$. 
Then 
$U$ has two boundaries whose numbers of tangencies are $c$ and $c'$. 
Let $v_c$ be a flow on an open disk $D_c$ with $c$ tangencies   
and 
$v_{c'}$ a flow on an open disk $D_{c'}$ with $c'$ tangencies as above. 
Fix two curves $\gamma_c$, $\gamma_{c'}$ contained in regular orbits of $v_c$, $v_{c'}$. 
Replace $\gamma_c$ (resp. $\gamma_{c'}$) with 
a saddle with 
two separatrices 
and 
one homoclinic separatrix 
which bounds a center disk. 
For instance, 
the former in Figure \ref{rot} is the case $c = 6$. 
Note that 
we must choose two center disks 
whose flow directions are opposite. 
Removing a small center disk in $D_{c}$ (resp. $D_{c'}$) 
and pasting the new two boundaries, 
the resulting flow can be considered as 
a flow on $U$ which is desired. 
Suppose that 
$b > 2$. 
%
%When $U'$ has at least three boundaries, 
Applying the above process for punctured disks at $ b -1$ times, 
%the number of the boundaries of $U$, 
we can obtain a desired flow. 

3. case ( $g = 0$ and $c_0 > 0$ and $b_p > 0$): 
%$\partial U$ contains periodic orbits): 
For $\partial U - \bigsqcup_j O_j$, 
construct a punctured disk $D'$ 
each of whose boundaries has tangencies 
as in the case 2. 
Replacing $b_p$ regular orbits in $D'$ with 
a saddle with two separatrices and 
one homoclinic separatrix which bounds a center disk 
and replacing centers with boundaries, 
we obtain a desired disk.  

4. case ( $g > 0$):  
%
%
%
%Finally, 
%consider the case with genus $g > 0$. 
%Recall that 
%$[\gamma] = 0 \in \pi_1(U, \partial U)$ 
%and one connected component of $\partial U - \partial \gamma$ 
%contains exactly one tangency 
%for any connected component $\gamma$ of the intersection of $U'$ 
%and an periodic orbit. 
%
%Consider a flow $v'$ on a punctured disk 
%such that $v'$ corresponds to $v$ near a \nbd of $\partial U'$.  
%Suppose that 
%$U$ is orientable. 
Assume that 
$b = 0$. 
Consider a sphere $D$ which consists of two centers 
and periodic orbits. 
%satisfies the boundary condition as the above cases. 
Fix $2g$ periodic orbits $\gamma_i$, $\gamma_i'$ contained in $2g$ regular orbits in $D$. 
Replace $\gamma_i$ (resp. $\gamma_i'$) with 
a homoclinic saddle connection.  
Note that 
%two separatrices of $\gamma_i$ are a same orbit 
%if $\gamma_i$ is a periodic orbit contained in $U$,  
%and that 
we need choose each pair of two center disks 
where  
the flow directions are opposite. 
Removing small center disks 
and pasting the boundaries of center disks for each $\gamma_i$ and $\gamma_i'$, 
the resulting flow is desired. 
Assume that 
$b \neq 0$. 
Consider a punctured disk $D$ as above cases which is desired except the genus condition. 
%satisfies the boundary condition as the above cases. 
Fix $2g$ curves $\gamma_i$, $\gamma_i'$ contained in $2g$ regular orbits in $D$. 
Replace $\gamma_i$ (resp. $\gamma_i'$) with 
a saddle with 
two separatrices 
and 
one homoclinic separatrix 
which bounds a center disk. 
Note that 
two separatrices of $\gamma_i$ are a same orbit 
if $\gamma_i$ is a periodic orbit contained in $U$,  
and that 
we need choose each pair of two center disks 
where  
the flow directions are opposite. 
Removing small center disks 
and pasting the boundaries of center disks for each $\gamma_i$ and $\gamma_i'$, 
the resulting flow is desired. 

Suppose that 
$U$ is non-orientable. 
Assume that 
$b = 0$. 
Consider a sphere $D$ which consists of two centers 
and periodic orbits. 
%satisfies the boundary condition as the above cases. 
Fix $g$ periodic orbits $\gamma_i$ contained in $g$ regular orbits in $D$. 
Replace $\gamma_i$  with a homoclinic saddle connection.  
%which bounds a center disk. 
%Note that 
%two separatrices of $\gamma_i$ are a same orbit 
%if $\gamma_i$ is a periodic orbit contained in $U$. 
%
Replacing small center disks with M\"obius bands which 
consists of periodic orbits, 
the resulting flow is desired. 
Assume that 
$b \neq 0$. 
Consider a punctured disk $D$ as above cases which is desired except the genus condition. 
Fix $g$ curves $\gamma_i$ contained in $g$ regular orbits in $D$. 
Replace $\gamma_i$  with a saddle with two separatrices 
and one homoclinic separatrix 
which bounds a center disk. 
Note that 
two separatrices of $\gamma_i$ are a same orbit 
if $\gamma_i$ is a periodic orbit contained in $U$. 
Replacing small center disks with M\"obius bands which 
consists of periodic orbits, 
the resulting flow is desired. 
\end{proof}

\begin{figure}\label{fig01}
\begin{center}
\includegraphics[scale=0.25]{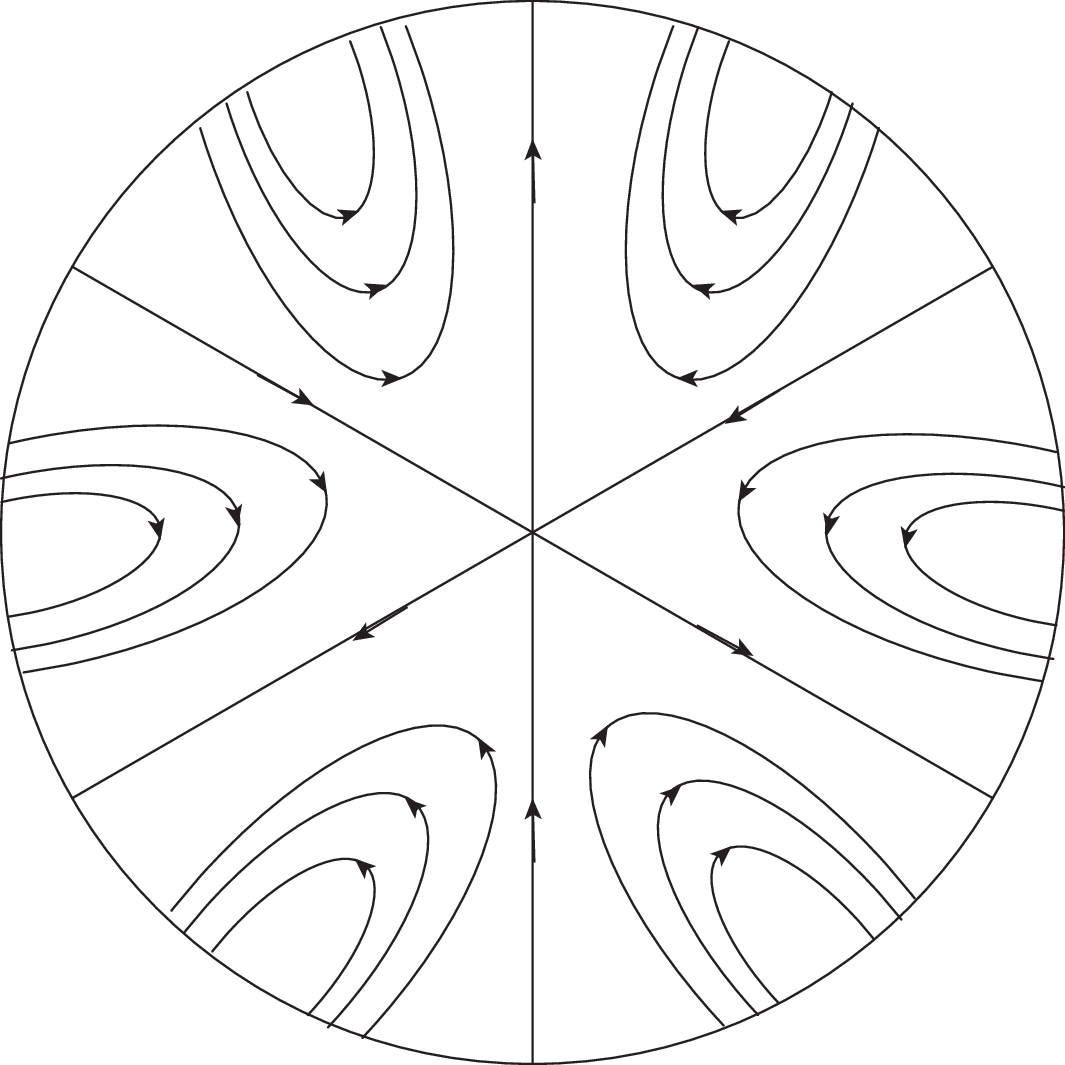}
\includegraphics[scale=0.35]{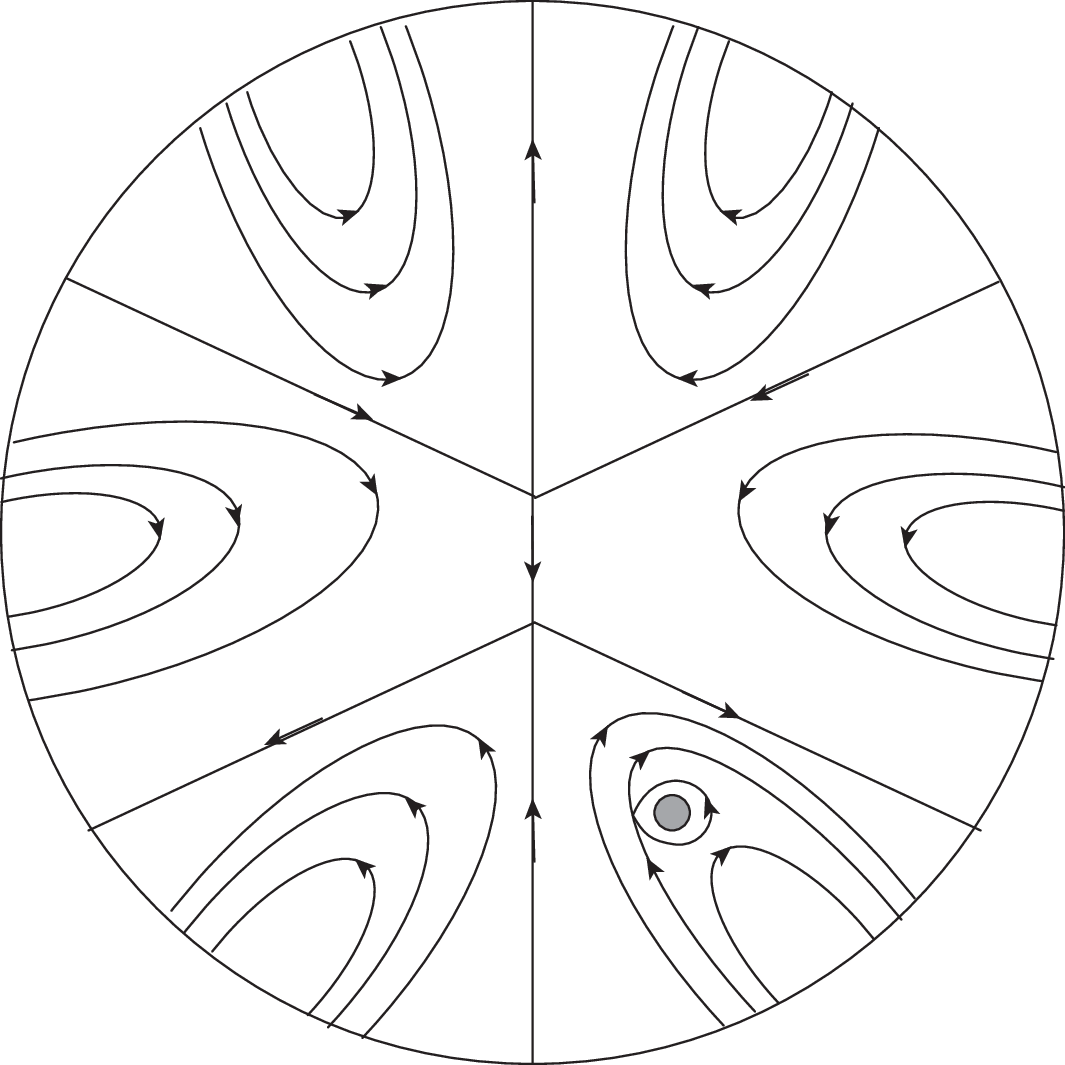}
\end{center}
\caption{A disk with a $2$-saddle and a disk with three $1$-saddles and one hole}
\label{rot}
\end{figure}

%\section{Main results}

Now we summarize the previous lemmas. 
%The previous lemmas imply the following result. 

\begin{theorem}\label{prop04}
Let $M$ be a closed surface.  
There is a dense subset $\mathcal{D}$ of $\chi^0_{\mathrm{nw}}$ 
such that 
$M - \mathop{\mathrm{Per}}(v)$ is the union of finitely many centers and 
finitely many homoclinic saddle connections 
for each $v \in \mathcal{D}$.  
\end{theorem}

The continuous structurally stability does not hold on 
the sphere $\S^2$ (resp. $\mathbb{P}^2$). 
%, \mathbb{K}^2$).  
Indeed, 
%if $r = 0$, then 
the regular singular point can be replaced 
by a closed ball with nonempty interior. 
However 
we will show that 
there are a few topologically stable flows on $\S^2$ (resp. $\mathbb{P}^2$) 
in the next section. 

\section{Topologically stability}

Recall that 
$v$ is topologically stable in $\chi^0_{\mathrm{nw}}$ if 
for any $w \in \chi_{nw}^0$ 
near $v$ with respect to the $C^0$-topology, 
there is a surjective continuous map 
such that 
the image of each $w$-orbit by it is an orbit. 
%which takes the orbits of $w$ onto the orbits of $v$. 
The surjective continuous map is called the 
semi-conjugacy. 
%A separatrix is an orbit whose $\omega$- limit set (resp. $\alpha$-limit set) 
%is a saddle point. 
To state the topologically stability of surface flows, 
we describe two lemmas.

\begin{lemma}\label{lem26}
Let $v, w$ be non-wandering continuous flows on a compact connected surface $M$ 
and $h : M \to M$ a semi-conjugacy 
which takes the orbits of $w$ onto the orbits of $v$. 
Suppose that 
$w$ is regular 
%and has 
%%neither heteroclinic connections 
%no locally dense orbits 
and 
that 
there are 
a point $x \in \mathop{\mathrm{Sing}}(v) \cap 
\overline{\mathop{\mathrm{Per}}(v)}$ 
and 
a $w$-saturated connected subset $W \subset \mathop{\mathrm{Per}}(w)$
such that 
$h(W)$ is a \nbd of $x$. 
Then $x$ is a center with respect to $v$. 
Moreover
if 
the image of a connected component $V$ of 
$\mathop{\mathrm{Per}}(w)$ 
contains at least two distinct centers with respect to $v$, 
then $h(V) = M$ is a sphere and 
$\mathop{\mathrm{Sing}}(v)$ consists of two centers. 
\end{lemma}

\begin{proof} 
Let $x$, $W$ as above. 
%Since the inverse image of a locally dense orbit by $h$ 
%is not closed,  
%there are no locally dense orbits of $v$ 
%and so 
%$\overline{\mathop{\mathrm{Per}}(v)} \supseteq M - \mathop{\mathrm{Sing}}(v)$. 
%
Since $x \in \overline{\mathop{\mathrm{Per}}(v)} \cap h(W)$, 
there are 
points $y, z \in W$ such that 
$h(y) = x$ and 
$h(z) \in \mathop{\mathrm{Per}}(v)$. 
Then 
there is an open saturated annulus $A \subset W$ 
such that 
$h(A)$ consists of periodic orbits of $v$, 
one of $\partial A$ is $O_w(z)$,  
and 
another of $\partial A$ is a periodic orbit of $w$ 
whose image by $h$ is a singular point of $v$.  
Note that 
$\overline{A}$ is a closed annulus 
such that 
$h(\overline{A}) = \overline{h(A)}$. 
Since $h(A)$ is a saturated annulus, 
the closure $\overline{h(A)}$ 
is a closed disk. 
%Suppose that 
If there is an open saturated annulus $A' \subset W$ 
such that 
$h(A')$ consists of periodic orbits of $v$, 
$\partial A'$ consists of two periodic orbits of $w$,  
and 
$h(\partial A')$ consists of two singular points of $v$, 
then the image $h(A')$ is a sphere 
which consists of two centers and periodic orbits. 
Thus 
we may assume that 
$h(\partial A')$ contains 
at most one singular point of $v$ 
for any  open saturated annulus $A' \subset W$ 
with $h(A') \subseteq \mathop{\mathrm{Per}}(v)$. 
Since $\overline{\mathop{\mathrm{Per}}(v)} \supseteq M - \mathop{\mathrm{Sing}}(v)$, 
we obtain that 
each connected component of $\mathop{\mathrm{Sing}}(v)$ 
intersects to $\overline{\mathop{\mathrm{Per}}(v)}$ 
and so 
the boundary of each connected component of $\mathop{\mathrm{Sing}}(v)$ in $h(W)$ 
is a center. 
This implies that 
each connected component of $\mathop{\mathrm{Sing}}(v)$ in $h(W)$ 
is a center. 
\end{proof}

%
%By  a $k$-saddle, 
%we mean a singular point  
%with $(2k + 2)$-separatrices. 

We state a non-existence of singular points except centers of topologically stable flows.

\begin{lemma}\label{lem3.3}
Let $v$ be a non-wandering continuous flow on a closed connected surface $M$. 
If $v$ is topologically stable in $\chi^0_{\mathrm{nw}}$, 
then $\mathop{\mathrm{Sing}}(v)$ consists of  
finitely many centers.  
\end{lemma}

\begin{proof} 
We may assume that $\mathop{\mathrm{Sing}}(v) \neq \emptyset$.  
Theorem \ref{prop04} implies that 
there are 
a regular flow $w \in \chi^0_{\mathrm{nw}}$ near $v$ 
and 
a semi-conjugacy 
$h: M \to M$ 
which takes the orbits of $w$ onto the orbits of $v$ 
such that 
$w$ has neither 
locally dense orbits 
nor heteroclinic connections. 
Since the inverse image of 
separatrices (resp. locally dense orbits) consists of 
separatrices (resp. locally dense orbits), 
the union $\mathrm{P}(v)$ consists of finitely many orbits 
and 
$v$ has neither 
locally dense orbits 
nor heteroclinic connections.  
%Then $\mathrm{P}(v) = M - (\mathop{\mathrm{Sing}}(v) \sqcup \mathop{\mathrm{Per}}(v))$ 
%consists of finitely many orbits.
%%
By Corollary 2.9 \cite{Y}, 
%we have that 
%each connected component of $\mathop{\mathrm{Per}}(w)$ 
%is an open annulus or a M\"obius band. 
%Since $M - \mathop{\mathrm{Per}}(w)$ 
%is a finite union of saddle connections 
%with respect to $w$, 
the subset $\mathop{\mathrm{Per}}(w)$ 
consists of 
finitely many annuli and M\"obius bands. 
Suppose that 
$v$ has infinitely many singular points. 
By Lemma \ref{lem26}, 
there are infinitely many centers with respect to $v$. 
Since $\mathop{\mathrm{Per}}(w)$ 
consists of finitely many annuli and M\"obius bands, 
there is a connected component of $\mathop{\mathrm{Per}}(w)$ 
whose image contains infinitely many centers. 
By lemma \ref{lem26}, 
we obtain 
$|\mathop{\mathrm{Sing}}(v)| = 2$, 
which contradicts.  
Thus $\mathop{\mathrm{Sing}}(v)$ is finite. 
By Theorem 3\cite{CGL}, 
each singular point is 
either a center or a multi-saddle. 
Finally, 
we will show that 
$v$ has no multi-saddles. 
Indeed, 
Lemma 2.4 \cite{Y} implies that 
$\overline{\mathop{\mathrm{Per}}(v)} = M$.  
Since $\mathop{\mathrm{Sing}}(v)$ is finite, 
each singular point is isolated. 
Suppose 
%there is a $k$-saddle $x$. 
%Assume 
there is a $0$-saddle $x$. 
%$x$ is a $0$-saddle. 
Removing $x$, 
we can obtain a flow $w$ near $v$ such that 
the multi-saddles (resp. the separatrices) 
of $v$ are fewer than 
those of $w$.  
%$|\mathop{\mathrm{Sing}}(w)| - |\mathop{\mathrm{Sing}}(v)| = 1$. 
This contradicts to the existence of a semi-conjugacy. 
Thus we may assume that there are no $0$-saddles of $v$.  
Suppose there is a $k$-saddle $x$ for some $k > 0$. 
%By Theorem 2.3 \cite{Y}, 
%For a saddle $x$, 
Then there is a connected component $A_x$ 
of $\mathop{\mathrm{Per}}(v)$ 
one of the 
whose boundaries 
contains at least two separatrices. 
For each connected component $A_y$ of $\mathop{\mathrm{Per}}(v)$ 
one of the whose boundaries contains at least two separatrices 
of a multi-saddle $y$, 
by any small perturbation, 
we can replace countably many periodic orbits $O'_n$ in $A_y$ 
%converging $y$ 
with homoclinic $0$-saddle connections $\{ y_n \} \sqcup O_n$ (where $O_n := O'_n - \{ y_n\}$)
such that 
$\lim_{n \to \infty} y_n = y$ and 
the $\omega$-limit set (resp. $\alpha$-limit set) of $O_n$ is $y_n$. 
%Then  for each multi-saddle $y$, 
%there is the sequence $(y_n)$ converges to $y$. 
Let  $w$ be the resulting flow. 
Note that 
there is no orbit-preserving continuous surjection  
from a periodic orbit to a multi-saddle. 
Fix a semi-conjugacy $h$ from $w$ to $v$. 
By the surjectivity of $h$, 
there are a $0$-saddle $z$ whose image of $h$ is a multi-saddle, 
and an open $w$-saturated annulus $A \subset \mathop{\mathrm{Per}}(w)$ 
whose boundary consists 
of a periodic orbit and 
of the homoclinic $0$-saddle connection of $z$   
such that 
one of connected components of $h(\partial A) =\partial (h(A))$ 
contains at least two separatrices of $h(z)$. 
Therefore 
the image of the homoclinic $0$-saddle connection of $z$ by $h$ 
contains at least two separatrices of $h(z)$, 
which is impossible.  
%, denoted by $\partial_- A$,   
%such that 
%$h(\partial_- A) \subset h(\partial A) =\partial (h(A))$ 
%there are $0$-saddles $z, z'$ whose images of $h$ are a multi-saddle, 
%and an open saturated annulus $A$ 
%whose boundary 
%is the union of homoclinic $0$-saddle connections of $z$ or $z'$ 
%such that 
%one of $h(\partial A) =\partial (h(A))$ 
%contains 
%%whose image contains 
%at least two separatrices of $h(z)$, 
%which is impossible.  
\end{proof}

We state the topological stability on the sphere $\S^2$ 
(resp. the projective plane $\mathbb{P}^2$, 
the Klein bottle $\mathbb{K}^2$).

\begin{proposition}\label{prop33}
Let $v$ be a non-wandering continuous flow on the sphere $\S^2$ 
(resp. the projective plane $\mathbb{P}^2$). 
Then 
$v$ is topologically stable in $\chi^0_{\mathrm{nw}}$  
if and only if 
$v$ consists of two centers (resp. one center) and periodic points. 
%neither heteroclinic connections nor contractible separatrices.  
\end{proposition}

\begin{proof} 
Suppose that 
$v$ consists of two centers (resp. one center) 
and periodic points. 
Taking any small perturbation, 
let $v'$ be the resulting non-wandering flow. 
Then  $\mathop{\mathrm{Per}}(v')$ is not empty 
and so we can easily construct a semi-conjugacy, 
by 
collapsing non-periodic orbits into centers.  
Conversely, 
suppose that 
 $v$ is topologically stable in $\chi^0_{\mathrm{nw}}$.   
%We will show that 
%$v$ is regular. 
%% and has no heteroclinic connections. 
%Indeed, 
Lemma \ref{lem3.3} implies that 
$\mathop{\mathrm{Sing}}(v)$ consists of  
finitely many centers.  
Poincar\'e-Hopf theorem implies that 
the singular points are just two centers (resp. one center). 
This implies that 
$v$ consists of two centers (resp. one center) and periodic points. 
\end{proof}

\begin{proposition} \label{prop34}
Let $v \in \chi^0_{\mathrm{nw}}$ be a flow on 
the Klein bottle $\mathbb{K}^2$. 
Then 
$v$ is topologically stable in $\chi^0_{\mathrm{nw}}$  
if and only if 
$v$ consists of periodic points. 
\end{proposition}

\begin{proof} 
Suppose that 
$v$ is topologically stable in $\chi^0_{\mathrm{nw}}$. 
By Lemma \ref{lem3.3}, 
Poincar\'e-Hopf theorem implies that 
$v$ consists of periodic points. 
Conversely, 
suppose that 
$v$ consists of periodic points. 
Then $\mathbb{K}^2$ is a union of annuli and M\"obius bands. 
%Notice that  
%a union of one saturated M\"obius band and one saturated annulus with one intersection 
%is a M\"obius band 
%and that 
%a union of two saturated annuli with one intersection 
%is an annulus. 
%Since an annulus (resp. a M\"obius band) has two boundaries (resp. one boundary), 
%we have that 
%$\mathbb{K}^2$ is a union of two saturated M\"obius bands. 
Thus 
each pointwise periodic flow on $\mathbb{K}^2$ 
is topologically equivalent. 
Since each flow near $v$ consists of periodic points, 
we obtain that 
$v$ is topologically stable in $\chi^0_{\mathrm{nw}}$.  
\end{proof}

Recall the orbit space $M/v$ is 
the quotient space $M/\sim_v$ of $M$ by 
an equivalent relation $\sim_v$, 
where $x \sim_v y $ if 
$O_v(x) = O_v(y)$.  
%
%The above propositions 
%implies that 
%the following corollary. 
%
We characterize the topological stability for 
non-wandering continuous flows 
on closed connected surfaces.

\begin{proposition} 
Let $M$ be a closed connected surface. 
A non-wandering continuous flow 
on $M$  
is topologically stable 
if and only if 
the orbit space of it is homeomorphic to 
a closed interval. 
In particular,  
$M$ has a topologically stable flow in $\chi^0_{\mathrm{nw}}$ 
if and only if 
$M$ is either $\S^2$, $\mathbb{P}^2$, or $\mathbb{K}^2$. 
\end{proposition}

\begin{proof}
Suppose that 
$M$ has a topologically stable flow in $\chi^0_{\mathrm{nw}}$.  
By Lemma \ref{lem3.3}, 
the surface $M$ is homeomorphic to 
either 
$\T^2$, $\S^2$, $\mathbb{P}^2$, or $\mathbb{K}^2$. 
Assume that $M \cong \T^2$. 
Then there is a topologically stable flow $v$ on $\T^2$ in $\chi^0_{\mathrm{nw}}$ 
which is topologically equivalent to 
a rational (resp. irrational) rotation.  
By arbitrary small perturbation, 
the flow $v$ can become topologically equivalent to 
an irrational (resp. rational) rotation, 
which contradicts to the topological stability.  
The converse holds by previous two propositions. 
Moreover 
%As above, 
the orbit space of a topologically stable non-wandering continuous flow 
on a closed connected surface is homeomorphic to 
a closed interval. 
Finally, 
%Conversely, 
suppose that 
the orbit space of a non-wandering continuous flow $v$ 
is homeomorphic to a closed interval. 
Assume that $\mathrm{P} \neq \emptyset$. 
Fix any point $x \in \mathrm{P}$. 
By Proposition 2.6 \cite{Y}, 
the closure $\overline{ O_v(x)}$  
contains a point $y \in \mathop{\mathrm{Sing}}(v)$ 
and so  
$M/v$ is not $T_1$, which contradicts to the hypothesis. 
Thus $\mathrm{P} = \emptyset$ 
and so 
$M = \mathop{\mathrm{Sing}}(v) \sqcup \mathop{\mathrm{Per}}(v)$. 
Since $M/v$ is a closed interval, 
we obtain 
$\mathrm{int} \mathop{\mathrm{Sing}}(v) = \emptyset$ 
and so 
$M = \overline{\mathop{\mathrm{Per}}(v)}$. 
First suppose that  $\mathop{\mathrm{Sing}}(v) = \emptyset$. 
Then $M = \mathop{\mathrm{Per}}(v)$ is either 
a torus or a Klein bottle. 
If $M$ is a torus, then $M/v \cong \S^1$, which contradicts to the hypothesis. 
Thus $M$ is a Klein bottle.  
%By Proposition \ref{prop34}, 
%the flow $v$ is topologically stable. 
Second suppose that  $\mathop{\mathrm{Sing}}(v) \neq \emptyset$. 
Assume that 
there is a connected component $U$ of $\mathop{\mathrm{Per}}(v)$ 
which is homeomorphic to a M\"obius band. 
Then $U$ in $M/v$ is homeomorphic to $[0,1)$ 
and $\overline{U} -U$ is a singular point. 
This means that 
$v$ consists of one center and periodic points
and so 
$M \cong \mathbb{P} ^2$.   
%By Proposition \ref{prop33}, 
%the flow $v$ is topologically stable. 
Assume that 
there is a connected component 
$U$ of $\mathop{\mathrm{Per}}(v)$ 
which is not homeomorphic to a M\"obius band. 
Then 
each connected component of $\overline{U} -U$ is a singular point. 
Since $M/v$ is a closed interval, 
the flow $v$ consists of two centers and periodic points 
and so
$M$ is a sphere.   
%By Proposition \ref{prop33}, 
%the flow $v$ is topologically stable. 
\end{proof}

%By Poincar\'e-Hopf theorem, 
%we can state the following statement. 
% implies that 
%there are no topologically stable non-wandering flows  
%in $\chi^0_{\mathrm{nw}}$ 
%on $M$. 
Summarize the above statements. 

\begin{theorem}\label{}
A closed connected surface 
$M$ has a topologically stable 
continuous non-wandering flow in $\chi^0_{\mathrm{nw}}$ 
if and only if 
$M$ is either 
$\S^2$, $\mathbb{P}^2$, or $\mathbb{K}^2$.  
Moreover 
the following are equivalent for 
a continuous non-wandering flow $v$ on $M$: 

1. $v$ is topologically stable in $\chi^0_{\mathrm{nw}}$. 

2. 
$M/v \cong [0,1]$. 

3. 
$M \neq \T^2$ consists of periodic orbits and centers. 
\end{theorem}

%
%
%
%On the other hand, 
Finally, 
we state the genericity of 
topologically stable non-wandering flows 
with locally dense orbits.

\begin{proposition} 
%Let $M$ be .  
%which is neither the projective plane nor the Klein bottle. 
%Then there are no topologically stable non-wandering flows  
%in $\chi^0_{\mathrm{nw}}(M)$.   
%%for any $r  = 0$ or $1$. 
%Moreover 
%if $M$ is orientable, then 
Each non-wandering flow on an orientable connected closed surface with positive genus 
%$M$ 
can be approximated by regular non-wandering flows with 
locally dense orbits and 
without heteroclinic connections. 
\end{proposition}

%The following proof is analogous to 
%the proof of  Theorem 3.1.4 \cite{MW}.  

\begin{proof} 
Let $v$ be a non-wandering flow on $M$. 
By Theorem \ref{prop04}, 
we may assume that 
$v$ is regular such that  
$M - \mathop{\mathrm{Per}}(v)$
%the complement of the set of periodic points 
is the union of finitely many centers and 
finitely many homoclinic saddle connections. 
By Corollary 2.9 \cite{Y}, 
each connected component of 
$\mathop{\mathrm{Per}}(v)$ is 
an annulus. 
By a topologically conjugate, 
we may assume that 
$v$ is $C^{\infty}$. 
By induction on the genus $g$ of $M$, 
we show the assertion. 
Suppose that 
$g = 1$. 
Then 
$M$ is a torus. 
Let  $\gamma'$ be an essential saddle connection. 
Since each connected component of 
the complement of $\gamma'$ is an annulus or a disk, 
we have that 
all the saddle connections are 
homologically linearly dependent 
to $\gamma'$ 
and that 
there is an essential periodic orbit $\gamma$ of $v$ 
which is homologically linearly dependent to $\gamma'$ 
(see Figure \ref{ss}).  
%either contractible or 
%homologous to $\gamma$ up to multiplicative nonzero constant 
%and 
%the complement of saddle connections consists of 
%boundary points and singular points. 
Then $M - \gamma$ is homeomorphic to an open annulus. 
Note that 
each essential homoclinic saddle connection is of form 
either $\mu$ or $\gamma'$ as in Figure \ref{ss}. 
Therefore the restriction $v|_{M - \gamma}$
is topologically conjugate to a flow of 
a Hamiltonian vector field.  
This implies that 
$v$ is topologically conjugate to a flow of 
a divergence-free vector field.  
Take a simple closed curve $C$ which 
intersects $\gamma$
and 
is homologically linearly independent 
to $\gamma$ 
such that $C$ meets 
neither centers nor center desks whose 
closures are contractible 
and that 
$C \cap M'$ is connected for any connected component $M'$ of $M - A$, 
where $A$ is the union of the saddle connections which bound
the center disks  each of whose closure is essential. 
Since  the sum of two divergence-free vector fields is divergence-free, 
we can construct 
a sufficiently slow flow $v'$ whose support is a small annular \nbd of $C$ 
and which consists of periodic orbits 
and singular points 
such that 
$v + v'$ is non-wandering with 
$\mathop{\mathrm{Sing}}(v) = \mathop{\mathrm{Sing}}(v + v')$ 
and the closure of each center disk for $v + v'$ is contractible. 
%Let $D$ be the union of open center disks bounded by $A$. 
%
%
%
%because $v$ is topologically conjugate to the flow of a divergence-free vector field  
%and the sum of two divergence-free vector fields is divergence-free.  
%the restriction of $v'$ on the intersection of center disks 
%is symmetric. 
%Then 
%each connected component of $D - \mathrm{Supp}(v')$ 
%containing a center with respect to $v$ 
%is a subset of a center disk of $v + v'$. 
%belongs to 
%%consists of closed orbits of 
%$\mathop{\mathrm{Per}}(v + v') \sqcup \mathop{\mathrm{Sing}}(v + v')$. 
We may assume that 
$v+ v'$ has no locally dense orbits. 
Let $D'$ be the finite union of 
center disks for $v + v'$. 
%Then the closure of each center disk for $v + v'$ is contractible. 
Collapsing each center disk for $v + v'$ into a point 
and removing the new $0$-saddles, 
the resulting flow is a rational rotation.  
%
%By construction, 
Therefore 
there is a closed transversal $\gamma''$ for $v + v'$ 
which is near and 
homotopic to $\gamma$ 
such that $\gamma'' \cap \overline{D'} = \emptyset$ 
and 
each essential periodic orbit of $v + v'$ in $\mathop{\mathrm{Per}}(v + v')$ 
intersects $\gamma''$.  
%and   
%whose regular orbits are periodic orbits. 
Then the first return map of $v + v'$ on $\gamma''$ is 
a rotation except at most finitely many points. 
By an arbitrary small perturbation along $\gamma''$, 
% by a rotation, 
we may assume that  
%$v + v'$ is non-wandering 
%and  
%By an arbitrary small perturbation, 
%we may assume that 
$v + v'$ 
contains locally dense orbits. 
%, 
%by choosing that  
%$\mathrm{supp}(v')$ is contained in 
%a small annular \nbd of $C$ 
%and that 
%the intersection of $\mathrm{supp}(v')$ and 
%center disks is contained in 
%a small \nbd of the saddle connections. 
This  implies that 
$v$ can be approximated by regular non-wandering flows with 
locally dense orbits and 
without heteroclinic connections. 
%contradicts to 
%the structurally stability of $v$. 
Suppose that $M$ has genus $g(M) > 1$. 
%Assume that $M$ is orientable. 
%
Since $v$ is regular and has a singular point, 
there is an essential saddle connection $C$. 
Since $\partial \mathop{\mathrm{Per}}(v) \supset \mathrm{P}$  
and $\mathop{\mathrm{Per}}(v)$ is the 
union of disjoint open annuli, 
there is an essential simple closed curve $\gamma \subseteq C$ 
through a saddle. 
% in  $C$. 
Adding two disks with 
%$k - l$ $\partial$-saddles and $l$ 
two $\partial$-$0$-saddles 
to the new two boundaries of $M - \gamma$, 
% for some $l$, 
we obtain the new surface $M'$ and 
the new non-wandering flow $w'$ on $M'$. 
Since the genus $g(M')$ is less than $g(M)$, 
by inductive hypothesis, 
$w'$ can be approximated by a regular non-wandering flow $w''$ with 
locally dense orbits and without heteroclinic connections. 
By above construction, 
$w'$ and $w''$ coincide on 
a small \nbd of the union of centers each of whose closures is contractible. 
%Then 
%we may assume that 
%the small \nbd contains the new two center disks. 
Therefore 
$w''$ can be lifted to the non-wandering flow on $M$ with 
locally dense orbits which approximates $v$. 
%
%
%Assume that $M$ is non-orientable. 
%Since $v$ is regular and has a singular point, 
%there is a one-sided essential simple closed curve $\gamma$ which is 
%either a periodic orbit or contained in a saddle connection. 
%%
%Suppose that $\gamma$ is a periodic orbit (resp. contained in a saddle connection). 
%Then there is a \nbd of $\gamma$ which is homeomorphic to a M\"obius band. 
%Adding a center disk (resp. a disk with $\partial$-saddles and $\partial$-$0$-saddles 
%as in Figure \ref{fig01}) 
%to the new boundary of $M - \gamma$, 
%we obtain the new surface $M'$ and 
%the new non-wandering flow $w'$ on $M'$. 
%Since the genus $g(M')$ is less than $g(M)$, 
%by inductive hypothesis, 
%%$w'$ can be approximated by a regular non-wandering flow $w''$ with 
%%locally dense orbits and without heteroclinic connections. 
%the above argument constructs 
%a desired non-wandering flow on $M$ with locally dense orbits which approximates $v$. 
%
This completes the proof. 
\end{proof}

\begin{figure}
\begin{center}
\includegraphics[scale=0.45]{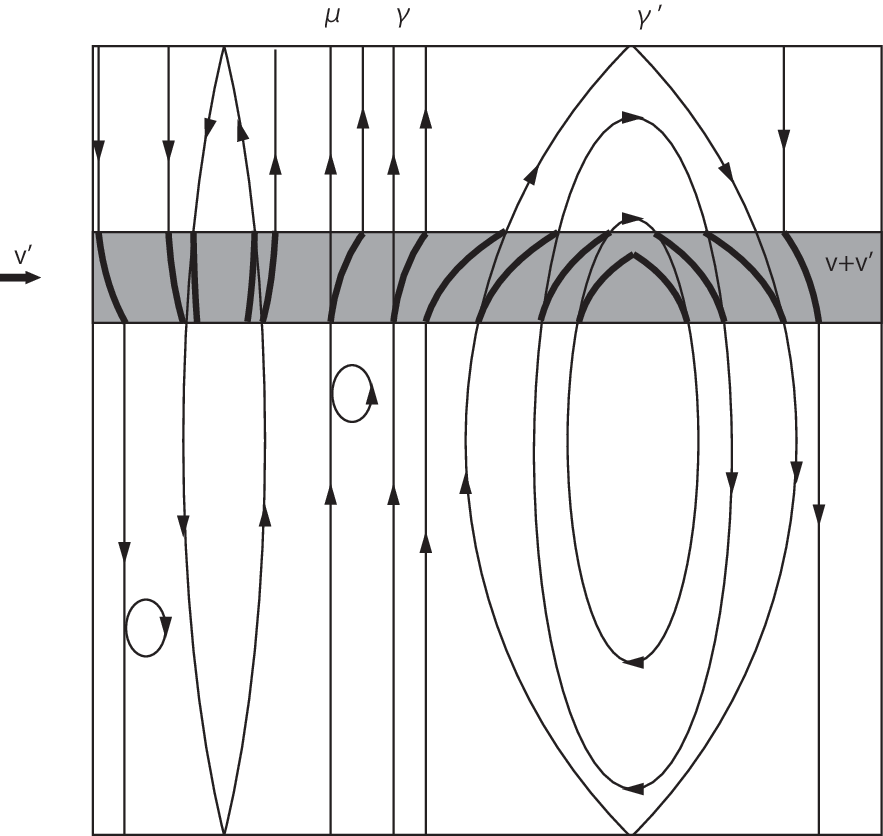}
\end{center}
\caption{Perturbation $v + v'$}
\label{ss}
\end{figure}

%\section{An example}
%%\section{Examples}
%
%We construct 
%a non-wandering flow on $\T^2$ 
%such that 
%$\mathrm{P}$ and 
%$\mathrm{LD}$ are dense. 
%In particular, 
%$\mathrm{LD}$ is not open. 
%This example also shows that 
%the finiteness condition in Lemma \ref{lem02} 
%is necessary. 
%
%
%
%\begin{example}\label{ex1}
%Consider an irrational rotation $v$ on $\T^2$. 
%Fix any points $p \in \T^2$ with $O_v(p)$. 
%Using dump functions, 
%replace $O_v(p)$ 
%with a union of countably many singular points $p_i$ ($i \in \Z$) 
%and countably many proper orbits 
%such that 
%$(p_i)$ converges to $p_0$ as $i \to \infty$ (resp. $i \to - \infty$).
%Moreover 
%we choose that  
%there is a sequence $(t_i)$ 
%such that 
%$p_i = O_v(t_i, p_0)$ and 
%that 
%$t_i \to \infty$ as $i \to \infty$ 
%(resp. $t_i \to - \infty$ as $i \to -\infty$). 
%Let $v'$ be the resulting vector field. 
%For any point $x \in \T^2 - O(p)$, 
%we have $O_v(x) = O_{v'}(x)$. 
%Moreover $O_v(p) - \mathop{\mathrm{Sing}}(v') = \mathrm{P}(v')$.  
%%is the union of proper orbits of $v'$. 
%\end{example}
%

%We construct 
%a smooth non-wandering flow on $\T^2$ 
%which has uncountably many singular points robustly 
%for the $C^2$-topology.  
%
%\begin{example}
%Consider a smooth function $f: \T^1 \to [-1,1]$ 
%such that  and $\mathop{\mathrm{Fix}(f) = \{0, 1/2}$ 
%$df/dy(0) > 0$.  
%Define a flow $v: \R \times \T^2 \to \T^2$ by $v_t(x, y) = (x + t \cdot f(y), y )$. 
%Then the vector field $V:= \partial v/\partial t = (f(y), 0)$ is smooth 
%such that $\partial V/\partial x = 0$ and 
%$\partial V/\partial y(0) = df/dy(0) > 0$. 
%\end{example}

\end{document}